\documentclass[preprint,12pt]{article}
\usepackage{hyperref}
\usepackage{amssymb,amsthm}
\usepackage{mathrsfs}
\usepackage{amsfonts, dsfont}
\usepackage{graphicx}
\usepackage{colortbl,dcolumn}
\usepackage{amsmath}
\usepackage{psfrag}
\usepackage{booktabs}
\usepackage{comment}
\usepackage{cite}
\usepackage{url}
\allowdisplaybreaks[4]
\numberwithin{equation}{section}
\usepackage[top=1.1in, bottom=1.2in, left=0.8in, right=0.8in]{geometry}

\newcommand{\R}{\mathbb{R}}
\newcommand{\N}{\mathbb{N}}
\newcommand{\E}{\mathbb{E}}
\renewcommand{\P}{\mathbb{P}}

\newcommand{\diff}{{\,\rm{d}}}

\newcommand{\F}{\mathcal{F}}

\newtheorem{theorem}{Theorem}[section]
\newtheorem{definition}[theorem]{Definition}
\newtheorem{lemma}[theorem]{Lemma}
\newtheorem{proposition}[theorem]{Proposition}

\newtheorem{assumption}[theorem]{Assumption}

\begin{document}
\title{Large deviations principle for stochastic delay differential equations with super-linearly growing coefficients\footnotemark[1]}
\author{
Diancong Jin$\,^\text{1,2}$,
Ziheng Chen$\,^\text{3,}$\footnotemark[2],
Tau Zhou$\,^\text{4,5}$
\\\footnotesize 1. School of Mathematics and Statistics, Huazhong University of Science and Technology, Wuhan 430074, China
\\\footnotesize 2. Hubei Key Laboratory of Engineering Modeling and Scientific Computing,
\\\footnotesize Huazhong University of Science and Technology, Wuhan 430074, China
\\\footnotesize 3. School of Mathematics and Statistics, Yunnan University, Kunming 650500, China
\\\footnotesize 4. Academy of Mathematics and Systems Science, Chinese Academy of Sciences, Beijing 100190, China
\\\footnotesize 5. School of Mathematical Sciences, University of Chinese Academy of Sciences, Beijing 100049, China
}

\date{}
\maketitle
\footnotetext{\footnotemark[1] This work was supported by National Natural Science Foundation of China (Nos. 11971488, 11971470, 12031020, 12026428 and 12171047), the National key R\&D Program of China under Grant NO. 2020YFA0713701 and the Fundamental Research Funds for the Central Universities 3004011142.}
%

\footnotetext{\footnotemark[2]
Corresponding author: Ziheng Chen.
Email addresses: diancongjin@lsec.cc.ac.cn(D. Jin),
czh@ynu.edu.cn(Z. Chen),
zt@lsec.cc.ac.cn(T. Zhou).}

\begin{abstract}
      We utilize the weak convergence method to establish the Freidlin--Wentzell large deviations principle (LDP) for stochastic delay differential equations (SDDEs) with super-linearly growing coefficients, which covers a large class of cases with non-globally Lipschitz coefficients. The key ingredient in our proof is the uniform moment estimate of the controlled equation, where we handle the super-linear growth of the coefficients by an iterative argument. Our results allow both the drift and diffusion coefficients of the considered equations to super-linearly grow not only with respect to the delay variable but also to the state variable. This work extends the existing results which develop the LDPs for SDDEs with super-linearly growing coefficients only with respect to the delay variable.


\end{abstract}

\section{Introduction}\label{sec:introduction}
The theory of large deviations is one of the most active topics in probability and statistics, which deals with the asymptotics of small probabilities on an exponential scale. It has extensive applications in communication networks, information theory, statistical mechanics, queueing systems and so on
(see, e.g., \cite{dembo2009large,chen2010random,
chen2021asymptotically,chen2020large,hong2020numerically,
chen2021large,hong2021numerical} and references therein).
As an important part of the theory of large deviations, the LDP for stochastic differential equations (SDEs) with small noise, also called the Freidlin--Wentzell LDP, has received much attention in recent years. It characterizes the probabilities that the pathways of SDEs deviate from the pathways of their corresponding deterministic equations when the intensity of the noises tends to zero.
The Freidlin--Wentzell  LDP for SDEs originates from the seminal work \cite{freidlin1984random} by Freidlin and Wentzell, and has been extensively studied
(see, e.g., \cite{scheutzow1984qualitative,
budhiraja2000variational,mohammed2006large,mo2013large,
chiarini2014large,bao2015large,
suo2020large} and references therein).

Our main interest in the present paper is to develop
the LDP for SDDEs with small noises. SDDEs or general stochastic systems with memory describe the stochastic processes whose behavior depends not only on their present state but also on their past history. Systems of such type are widely used to model processes in physics, economy, finance, biology, medicine, etc. For the cases of Lipschitz continuous and linearly growing coefficients,
\cite{scheutzow1984qualitative} studied the LDP for SDDEs with additive noise  in 1984. To handle the case of multiplicative noise, \cite{mohammed2006large} used the classical discretization method to develop the LDP for SDDEs driven by
small multiplicative noise. Subsequently, both \cite{mo2013large} and \cite{chiarini2014large} employed the weak convergence method to establish  the LDP for SDDEs with small multiplicative noise, under certain mild conditions.
Taking into account that most of the models of applicable interest have super-linear growth coefficients, the linear growth condition on the coefficients becomes a significant limitation.
Recently, \cite{bao2015large} and \cite{suo2020large} obtained the LDPs for SDDEs with constant delay and general delay, respectively, both of which allow the coefficients to grow super-linearly with respect to the delay variable.
However, to the best of our knowledge, there are not any results about the LDP for SDDEs with super-linearly growing coefficients with respect to both the state variable and the delay variable (e.g., the stochastic delay power logistic model in \cite{mao2005khasminskii}).
This motivates us to make a contribution to this problem.

This work focuses on the following non-autonomous SDDE
\begin{equation}\label{eq:SDDE}
\begin{split}
      \diff{X^{\varepsilon}}(t)
      =&~
      b(t,X^{\varepsilon}(t),X^{\varepsilon}(t-\tau))\diff{t}
      +
      \sqrt{\varepsilon}
      \sigma(t,X^{\varepsilon}(t),X^{\varepsilon}(t-\tau))
      \diff{W(t)},
      \quad t \in (0,T],
      \\
      X^{\varepsilon}(t) =&~ \phi(t),
      \quad t \in [-\tau,0],
\end{split}
\end{equation}
where the small parameter $\varepsilon > 0$ denotes the intensity of the noise and $\phi \in C([-\tau,0];\R^{d})$ with $\tau > 0$.
Moreover, $\{W(t)\}_{t \in [0,T]}$ is an $m$-dimensional standard Brownian motion on the complete filtered probability space $(\Omega,\F,\{\F_{t}\}_{t \in [0,T]},\P)$, where $\{\F_{t}\}_{t \in [0,T]}$ satisfies the usual conditions.
Here, the measurable functions $b \colon [0,T] \times \R^{d} \times \R^{d} \to \R^{d}$ and $\sigma \colon [0,T] \times \R^{d} \times \R^{d} \to \R^{d \times m}$ satisfy a globally monotone condition and a polynomial growth condition (see Assumption \ref{eq:mainass}), which allow these two functions to grow super-linearly with respect to both the state variable and the delay variable. One classical approach to establishing the  Freidlin--Wentzell LDPs for stochastic systems 
is by means of the time discretization argument and the contraction principle, which requires certain subtle exponential estimates of small probabilities; see, e.g., \cite{freidlin1984random} for more details.  Another approach to deriving the LDPs for stochastic systems is the weak convergence method, introduced by Dupuis and Ellis in \cite{dupuis1997weak}.
The weak convergence method, based on the variational representation of exponential functionals of Brownian motions (see, e.g., \cite{boue1998variational,budhiraja2000variational,
budhiraja2008large}), is used to validate the Laplace principle which is equivalent to the LDP if the underlying space is Polish.
One  merit of the weak convergence method is that the LDPs can be derived under weaker assumptions on the coefficients of the considered stochastic systems, compared with the discretization argument.

Here we obtain the LDP of \eqref{eq:SDDE} by means of the weak convergence method.
The key ingredient of the proof lies in the qualitative properties of the controlled equation \eqref{eq:Yvaruvart} and the skeleton equation \eqref{eq:skeletoneq}, including the tightness for the family of distributions of solutions to the controlled equations
and the continuity of the solution mapping of the skeleton equation.
The main difficulty we are faced with is the uniform moment estimate of the solution to the controlled equation.
Due to the super-linear growth of the coefficients and the low regularity of the controlled process, the usual methods used to derive the moment estimate of the original equation are not applicable.
We present that this difficulty can be overcome by taking advantage of a stochastic Gronwall lemma and an iterative argument.
With this preparation, we derive the Laplace principle, actually the LDP, of the SDDE \eqref{eq:SDDE} via verifying a  criterion of LDP proposed by \cite{budhiraja2000variational}. These results remove the restriction that the coefficients are only allowed to super-linearly grow with respect to the delay variable in the existing literature (see, e.g., \cite{bao2015large,suo2020large}).

The rest of this paper is organized as follows. Section \ref{sec:preliminary} contains some basic knowledge relevant to the theory of large deviations. In Section \ref{sec:result}, we establish the LDP for SDDEs with super-linearly growing coefficients in Theorem \ref{thm:ldpresult}, whose proof is postponed in Section \ref{sec:proof}.

\section{Preliminaries}\label{sec:preliminary}
In this section we give some standard definitions and results from the theory of large deviations. To this end, we begin with some notations. Let $\langle \cdot,\cdot\rangle$ and $|\cdot|$ be the Euclidean inner product and the corresponding norm in $\R^{d}$, respectively.
If $A$ is a vector or a matrix, its transpose is denoted by $A^{*}$,
and its trace norm is $|A| := \sqrt{\rm{trace}(A^{*}A)}$.
For any $a,b \in \R$, let $a \vee b := \max\{a,b\}$ and $a \wedge b := \min\{a,b\}$.
For any $[a,b] \subset \R$, let $L^{2}([a,b];\R^{d})$ stand for the space of all square integrable functions from $[a,b]$ to $\R^{d}$, and let $C([a,b];\R^{d})$ the space of continuous functions $f \colon [a,b] \to \R^{d}$, equipped with the supremum norm $|\cdot|_{C([a,b];\R^{d})}$.
For any $c \in (a,b)$ and any $\phi \in C([a,c];\R^{d})$, we set the Banach space $C_{\phi}([a,b];\R^{d}) := \{f \in C([a,b];\R^{d}) : f(t) = \phi(t), t \in [a,c]\}$ with the supremum norm denoted by $|\cdot|_{C_{\phi}([a,b];\R^{d})}$.
For two random variables $X$ and $Y$, we use the notation $X \overset{d}{=} Y$ to denote that they are identically distributed.
The notation $\xrightarrow[\varepsilon \to 0]{d}$ means the convergence in distribution for a family of random variables as $\varepsilon$ tends to $0$. By convention, the infimum of an empty set is interpreted as $+\infty$. For simplicity, denote by $C(a_{1},\cdots,a_{m})$ a generic positive constant depending on parameters $a_{1},\cdots,a_{m}$ that may vary for each appearance.

Let $\varepsilon > 0$ be an index parameter and $\{Z^{\varepsilon}\}_{\varepsilon > 0}$ a family of random variables from the probability space $(\Omega,\F,\P)$ to a Polish space (i.e., a complete separable metric space) $\mathcal{X}$. The following are the definitions of the rate functions and LDP; see, e.g., \cite{dupuis1997weak}.

\begin{definition}
      A function $I \colon \mathcal{X} \to [0,+\infty]$ is said to be a rate function, if for each $\alpha \in [0,+\infty)$, the level set
      $\{x \in \mathcal{X} : I(x) \leq \alpha\}$
      is a compact subset of $\mathcal{X}$.
\end{definition}

\begin{definition}
      Let $I$ be a rate function on $\mathcal{X}$.
      The family $\{Z^{\varepsilon}\}_{\varepsilon > 0}$ is said to satisfy the LDP on $\mathcal{X}$ with rate function $I$ if the following two conditions hold.
      \begin{enumerate}
        \item [(a)] Large deviation upper bound. For each closed subset $F$ of $\mathcal{X}$,
            $$\varlimsup_{\varepsilon \to 0}\varepsilon
            \log\P(Z^{\varepsilon} \in F)
            \leq
            -\inf_{x \in F}I(x).$$

        \item [(b)] Large deviation lower bound. For each open subset $G$ of $\mathcal{X}$,
            $$\varliminf_{\varepsilon \to 0}\varepsilon
            \log\P(Z^{\varepsilon} \in G)
            \geq
            -\inf_{x \in G}I(x).$$
      \end{enumerate}
\end{definition}

One classical result is that the LDP is equivalent to the Laplace principle; see, e.g., \cite[Theorems 1.2.1 and 1.2.3]{dupuis1997weak}.

\begin{definition}
      Let $I$ be a rate function on $\mathcal{X}$.
      The family $\{Z^{\varepsilon}\}_{\varepsilon > 0}$ is said to satisfy the Laplace principle on $\mathcal{X}$ with rate function $I$ if for all bounded continuous functions $h \colon \mathcal{X} \to \R$,
      \begin{equation*}
            \lim_{\varepsilon \to 0}\varepsilon\log
            \E\big[\exp\big(-\tfrac{h(Z^{\varepsilon})}
            {\varepsilon}\big)\big]
            =
            -\inf_{x \in \mathcal{X}}\{h(x) + I(x)\}.
      \end{equation*}
\end{definition}

\begin{proposition}\label{pro:ldplp}
      The family $\{Z^{\varepsilon}\}_{\varepsilon > 0}$ satisfies the LDP on $\mathcal{X}$ with rate function $I$ if and only if $\{Z^{\varepsilon}\}_{\varepsilon > 0}$ satisfies the Laplace principle on $\mathcal{X}$ with the same rate function $I$.
\end{proposition}

In view of this equivalent result, we will focus on the Laplace principle.
To present a criteria for the Laplace principle, we set
$$S_{\alpha} := \Big\{\varphi \in L^{2}([0,T];\R^{m}) \Big|
\int_{0}^{T}|\varphi(s)|^{2}\diff{s} \leq \alpha\Big\}$$
and
$$\mathcal{A}_{\alpha} :=
\big\{u \colon \Omega \times [0,T] \to \R^{m}~|~ u\text{~is~} \{\F_{t}\}_{t \in [0,T]}\text{-predictable and~}
u \in S_{\alpha}, \P\text{-a.s.}\big\}$$
for each $\alpha \in (0,\infty)$.
Throughout this paper, the set $S_{\alpha}$ will be always endowed with the weak topology of $L^{2}([0,T];\R^{m})$. Thus $S_{\alpha}$ is a compact Polish space under this weak topology.
The following criteria for the Laplace principle is due to Budhiraja and Dupuis in \cite[Theorem 4.4]{budhiraja2000variational}.

\begin{lemma}\label{lem:sufficientlemma}
      For each $\varepsilon > 0$, let $F^{\varepsilon} \colon C([0,T];\R^{m}) \to C_{\phi}([-\tau,T];\R^{d})$ be a measurable map. If there exists a measurable map $F \colon C([0,T];\R^{m}) \to C_{\phi}([-\tau,T];\R^{d})$ such that the following two conditions hold.
      \begin{enumerate}
        \item [(a)] For each $\alpha \in (0,\infty)$, the set
                    $$\Big\{F\Big(\int_{0}^{\cdot} \varphi(s) \diff{s}\Big) \Big| \varphi \in S_{\alpha}\Big\}$$
                    is a compact subset of $C_{\phi}([-\tau,T];\R^{d})$.

        \item [(b)] If
                    $\{u^{\varepsilon}\}_{\varepsilon > 0} \subset \mathcal{A}_{\alpha}$ for some $\alpha \in (0,\infty)$ and $u^{\varepsilon} \xrightarrow[\varepsilon \to 0]{d} u$ as $S_{\alpha}$-valued random variables, then $$F^{\varepsilon}\Big(W + \frac{1}{\sqrt{\varepsilon}} \int_{0}^{\cdot} u^{\varepsilon}(s) \diff{s}\Big) \xrightarrow[\varepsilon \to 0]{d} F\Big(\int_{0}^{\cdot} u(s) \diff{s}\Big).$$
      \end{enumerate}
      Then $\{F^{\varepsilon}(W)\}_{\varepsilon > 0}$ satisfies the Laplace principle (hence LDP) on $C_{\phi}([-\tau,T];\R^{d})$ with rate function $I \colon C_{\phi}([-\tau,T];\R^{d}) \to [0,+\infty]$ defined by
      $$I(f)
      =
      \inf_{\{\varphi \in L^{2}([0,T];\R^{m}) |
      f = F(\int_{0}^{\cdot} \varphi(s) \diff{s})\}}
      \Big\{
      \tfrac{1}{2}\int_{0}^{T}|\varphi(s)|^{2}\diff{s}\Big\}.$$
\end{lemma}

\section{The LDP for SDDEs}\label{sec:result}
This section is devoted to formulating the LDP of \eqref{eq:SDDE} by using Lemma \ref{lem:sufficientlemma}.
To begin, let us make the following assumption concerning \eqref{eq:SDDE}.
\begin{assumption}\label{eq:mainass}
      Let $|b(t,0,0)| + |\sigma(t,0,0)| \leq K_{1}$ for all $t \in [0,T]$ and some constant $K_{1} > 0$.
      Also let $b$ and $\sigma$ satisfy the globally monotone condition, i.e., there exists constants $\eta \in (1,\infty)$ and $K_{2} > 0$ such that
      \begin{equation}\label{eq:couplemonocond}
      \begin{split}
            \langle x_{1}-x_{2},
            b(t,x_{1},y_{1})-&~b(t,x_{2},y_{2}) \rangle
            +
            \eta|\sigma(t,x_{1},y_{1})-\sigma(t,x_{2},y_{2})|^{2}
            \\\leq&~
            K_{2}(|x_{1}-x_{2}|^{2} + |y_{1}-y_{2}|^{2}),
            \quad t \in [0,T],x_{1},x_{2},y_{1},y_{2} \in \R^{d}.
      \end{split}
      \end{equation}
      Further, let $b$ satisfy the polynomial growth condition, i.e., there exist constants $q \geq 1$ and $K_{3} > 0$ such that
      \begin{equation}\label{eq:bpolynomialgrow}
      \begin{split}
            |b(t,x_{1},y_{1})-b(t,x_{2},y_{2})|
            \leq&~
            K_{3}(1+|x_{1}|^{q-1}+|x_{2}|^{q-1}+|y_{1}|^{q-1}+|y_{2}|^{q-1})
            \\&~\times (|x_{1}-x_{2}|+|y_{1}-y_{2}|),
            \quad t \in [0,T],x_{1},x_{2},y_{1},y_{2} \in \R^{d}.
      \end{split}
      \end{equation}
\end{assumption}

Since the case of \eqref{eq:bpolynomialgrow} with $q = 1$ coincides with the well-known globally Lipschitz case studied in \cite{mohammed2006large,chiarini2014large,mo2013large}, we will focus on the case of $q > 1$, i.e., the super-linear growing case.
For example, one can consider the scalar SDDE
\begin{equation*}
      \diff{X^{\varepsilon}(t)}
      =
      \big(X^{\varepsilon}(t)-X^{\varepsilon}(t)^{3}
      +X^{\varepsilon}(t-\tau)\big)\diff{t}
      +
      \tfrac{\sqrt{\varepsilon}}{2}
      X^{\varepsilon}(t)^{2}\diff{W(t)},
      \quad t \in [0,T],
\end{equation*}
where $\{W(t)\}_{t \in [0,T]}$ is a real valued standard Brownian motion. Moreover, we refer to, e.g., \cite{mao2005khasminskii} and references therein for more concrete SDDEs satisfying Assumption \ref{eq:mainass}.

For convenience, we note several consequences of Assumption \ref{eq:mainass}. It follows from $|b(t,0,0)| + |\sigma(t,0,0)| \leq K_{1}, t \in [0,T]$  and \eqref{eq:couplemonocond} that there exists $K_{4} > 0$ such that
\begin{equation}\label{eq:globalcoercivitycond}
      \langle x,b(t,x,y) \rangle
      +
      \tfrac{\eta}{2}|\sigma(t,x,y)|^{2}
      \leq
      K_{4}(1 + |x|^{2} + |y|^{2}),
      \quad t \in [0,T],x,y \in \R^{d}.
\end{equation}
Moreover, using \eqref{eq:couplemonocond} and \eqref{eq:bpolynomialgrow}, it is easy to show that there exists $K_{5} > 0$ such that
\begin{equation}\label{eq:sigmapolynomialgrow}
\begin{split}
      |\sigma(t,x_{1},y_{1})-\sigma(t,x_{2},y_{2})|
      \leq&~
      K_{5}(1+|x_{1}|^{q-1}+|x_{2}|^{q-1}+|y_{1}|^{q-1}+|y_{2}|^{q-1})
      \\&~\times (|x_{1}-x_{2}|+|y_{1}-y_{2}|),
      \quad t \in [0,T],x_{1},x_{2},y_{1},y_{2} \in \R^{d}.
\end{split}
\end{equation}
Finally, \eqref{eq:bpolynomialgrow} and \eqref{eq:sigmapolynomialgrow} give the super-linearly growing bound, i.e., there exists $K_{6} > 0$ such that
\begin{equation}\label{eq:bsigmasuperlinear}
\begin{split}
      |b(t,x,y)| + |\sigma(t,x,y)|
      \leq K_{6}(1+|x|^{q}+|y|^{q}),
      \quad t \in [0,T],x,y \in \R^{d}.
\end{split}
\end{equation}

According to \eqref{eq:bpolynomialgrow} and \eqref{eq:sigmapolynomialgrow}, we know that $b$ and $\sigma$ are locally Lipschitz continuous. This together with \eqref{eq:globalcoercivitycond} is sufficient
to ensure the existence of a unique strong solution to \eqref{eq:SDDE}; see, e.g., \cite[Theorem 1.2]{mao2005khasminskii}.
\begin{theorem}
      Suppose that Assumption \ref{eq:mainass} holds. Then for each $\varepsilon > 0$, \eqref{eq:SDDE} admits a unique strong solution $\{X^{\varepsilon}(t)\}_{t \in [-\tau,T]}$, described by $X^{\varepsilon}(t) = \phi(t)$ for all $t \in [-\tau,0]$ and
      \begin{equation*}
            X^{\varepsilon}(t)
            =
            \phi(0)
            +
            \int_{0}^{t}
                b(s,X^{\varepsilon}(s),X^{\varepsilon}(s-\tau))
            \diff{s}
            +
            \int_{0}^{t}
            \sqrt{\varepsilon}
            \sigma(s,X^{\varepsilon}(s),X^{\varepsilon}(s-\tau))
            \diff{W(s)},
            \quad t \in [0,T].
      \end{equation*}
\end{theorem}

To apply Lemma \ref{lem:sufficientlemma}, we will specify the maps $F^{\varepsilon}$ and $F$ in the context of SDDE \eqref{eq:SDDE}. For each $\varepsilon > 0$, it follows from $\{X^{\varepsilon}(t)\}_{t \in [-\tau,T]}$ being a strong solution to \eqref{eq:SDDE} and
the Yamada--Watanabe theorem (see, e.g., \cite{klenke2008probability})
that there exists a Borel measurable map
\begin{equation}\label{eq:mapFvarepsilon}
      F^{\varepsilon}
      \colon C([0,T];\R^{m}) \to C_{\phi}([-\tau,T];\R^{d})
\end{equation}
such that $X^{\varepsilon} = F^{\varepsilon}(W),\P$-a.s.
Similar to \cite[Lemma 1]{chiarini2014large},
for any $u^{\varepsilon} \in \mathcal{A}_{\alpha}$ with $\varepsilon > 0$ and $\alpha \in (0,\infty)$, the Girsanov theorem ensures that the stochastic process
\begin{equation}\label{eq:Yvaruvartttt}
      Y^{\varepsilon,u^{\varepsilon}}(t)
      :=
      F^{\varepsilon}\Big(W + \frac{1}{\sqrt{\varepsilon}}
\int_{0}^{\cdot} u^{\varepsilon}(s) \diff{s}\Big)(t),
\quad t \in [-\tau,T]
\end{equation}
is the unique strong solution of the following stochastic controlled equation
      \begin{equation}\label{eq:Yvaruvart}
      \begin{split}
            \diff{Y^{\varepsilon,u^{\varepsilon}}}(t)
            =&~
            b(t,Y^{\varepsilon,u^{\varepsilon}}(t),
            Y^{\varepsilon,u^{\varepsilon}}(t-\tau))\diff{t}
            +
            \sigma(t,Y^{\varepsilon,u^{\varepsilon}}(t),
            Y^{\varepsilon,u^{\varepsilon}}(t-\tau))
            u^{\varepsilon}(t)\diff{t}
            \\&~+
            \sqrt{\varepsilon}
            \sigma(t,Y^{\varepsilon,u^{\varepsilon}}(t),
            Y^{\varepsilon,u^{\varepsilon}}(t-\tau))
            \diff{W(t)},
            \quad t \in (0,T],
            \\
            Y^{\varepsilon,u^{\varepsilon}}(t) =&~ \phi(t), \quad t \in [-\tau,0].
      \end{split}
      \end{equation}
To define the map $F$, we introduce the skeleton
equation associated to \eqref{eq:SDDE} as follows
\begin{equation}\label{eq:skeletoneq}
\begin{split}
      \frac{\diff{z^{\varphi}(t)}}{\diff{t}}
      =&~
      b(t,z^{\varphi}(t),z^{\varphi}(t-\tau))
      +
      \sigma(t,z^{\varphi}(t),z^{\varphi}(t-\tau))\varphi(t),
      \quad t \in (0,T],
      \\
      z^{\varphi}(t) =&~ \phi(t), \quad t \in [-\tau,0]
\end{split}
\end{equation}
with $\varphi \in L^{2}([0,T];\R^{m})$; see Lemma \ref{lem:skeletoneqexu} for its well-posedness. Then one can define
\begin{equation}\label{eq:mapF}
      F \colon C([0,T];\R^{m}) \to C_{\phi}([-\tau,T];\R^{d}),
      f \mapsto F(f) = g
\end{equation}
with $g = z^{\varphi}$ if there exists $\varphi \in L^{2}([0,T];\R^{m})$ such that $f(\cdot) = \int_{0}^{\cdot}\varphi(s)\diff{s}$, otherwise $g \equiv 0$. Obviously, $F\big(\int_{0}^{\cdot} \varphi(s) \diff{s}\big) = z^{\varphi}$ for any $\varphi \in L^{2}([0,T];\R^{m})$.
After these preparations, we now elaborate our main result on the LDP for SDDE \eqref{eq:SDDE}, whose proof is postponed in Section \ref{sec:proof}.

\begin{theorem}\label{thm:ldpresult}
      Suppose that Assumption \ref{eq:mainass} holds with $\eta > 2q-1$.
      Then the family $\{X^{\varepsilon}\}_{\varepsilon > 0}$ satisfies the LDP on $C_{\phi}([-\tau,T];\R^{d})$ with rate function $I \colon C_{\phi}([-\tau,T];\R^{d}) \to [0,+\infty]$ defined by
      \begin{align*}
            I(f)
            =
            \inf_{\{
            \varphi \in L^{2}([0,T];\R^{m}) |
            f(\cdot) = \phi(0)
            +
            \int_{0}^{\cdot} b(s,f(s),f(s-\tau)) \diff{s}
            +
            \int_{0}^{\cdot}\sigma(s,f(s),f(s-\tau))\varphi(s)\diff{s}\}}
            \Big\{\frac{1}{2}\int_{0}^{T}|\varphi(s)|^{2}\diff{s}\Big\}.
      \end{align*}
\end{theorem}

\section{Proof of Theorem \ref{thm:ldpresult}}\label{sec:proof}

In this section, we intend to finish the proof of Theorem \ref{thm:ldpresult}. We first present several auxiliary results to characterize some qualitative properties of \eqref{eq:Yvaruvart} and \eqref{eq:skeletoneq}. The following lemma provides a priori estimate
for the solution of \eqref{eq:skeletoneq}.

\begin{lemma}\label{lem:squarezvarphit}
      Suppose that Assumption \ref{eq:mainass} holds and let $\{z^{\varphi}(t)\}_{t \in [-\tau,T]}$ be given by \eqref{eq:skeletoneq} with $\varphi \in L^{2}([0,T];\R^{m})$.
      Then we have
      \begin{equation}\label{eq:zvarphit2}
            |z^{\varphi}|_{C_{\phi}([-\tau,T];\R^{d})}^{2}
            \leq
            \big(|\phi(0)|^{2} + 2K_{4}T\big)
            \exp\big(4K_{4}T +
            \eta^{-1}|\varphi|_{L^{2}([0,T];\R^{m})}^{2}\big).
      \end{equation}
\end{lemma}

\begin{proof}
      Applying the Schwarz inequality, the weighted Young inequality $ab \leq \frac{\eta}{2} a^{2} + \frac{b^{2}}{2\eta}$ for any $a,b \in \R$ and \eqref{eq:globalcoercivitycond} yields
      \begin{equation*}
      \begin{split}
            \frac{\diff{|z^{\varphi}(t)|^{2}}}{\diff{t}}
            =&~
            2\langle z^{\varphi}(t),
            b(t,z^{\varphi}(t),z^{\varphi}(t-\tau))\rangle
            +
            2\langle z^{\varphi}(t),
            \sigma(t,z^{\varphi}(t),
            z^{\varphi}(t-\tau))\varphi(t)\rangle
            \\\leq&~
            2\big(\langle z^{\varphi}(t),
            b(t,z^{\varphi}(t),z^{\varphi}(t-\tau))\rangle
            +
            \tfrac{\eta}{2}|\sigma(t,z^{\varphi}(t),
            z^{\varphi}(t-\tau))|^{2}\big)
            +
            \tfrac{1}{\eta}
            |z^{\varphi}(t)|^{2}|\varphi(t)|^{2}
            \\\leq&~
            2K_{4}\big( 1+ |z^{\varphi}(t)|^{2} +
            |z^{\varphi}(t-\tau)|^{2}\big)
            +
            \tfrac{1}{\eta}
            |z^{\varphi}(t)|^{2}|\varphi(t)|^{2}.
      \end{split}
      \end{equation*}
      It follows that
      \begin{align*}
            |z^{\varphi}(t)|^{2}
            \leq&~
            |\phi(0)|^{2} + 2K_{4}T
            +
            \int_{0}^{t}
            \sup_{r \in [-\tau,s]}|z^{\varphi}(r)|^{2}
            \big(4K_{4}+\tfrac{1}{\eta}|\varphi(s)|^{2}\big)\diff{s}
      \end{align*}
      and thus
      \begin{equation*}
            \sup_{r \in [-\tau,t]}|z^{\varphi}(r)|^{2}
            \leq
            |\phi(0)|^{2} + 2K_{4}T
            +
            \int_{0}^{t}
            \sup_{r \in [-\tau,s]}|z^{\varphi}(r)|^{2}
            \big(4K_{4} +\tfrac{1}{\eta}|\varphi(s)|^{2}\big)\diff{s}.
      \end{equation*}
      Utilizing the Gronwall inequality yields that for all $t \in [0,T]$, one has
      \begin{equation*}\label{eq:zvarphirpgeq2}
      \begin{split}
            \sup_{r \in [-\tau,t]}|z^{\varphi}(r)|^{2}
            \leq&~
            \big(|\phi(0)|^{2} + 2K_{4}T\big)
            \exp\Big(\int_{0}^{t}4K_{4} +
            \tfrac{1}{\eta}|\varphi(s)|^{2}\diff{s}\Big)
            \\\leq&~
            \big(|\phi(0)|^{2} + 2K_{4}T\big)
            \exp\Big(4K_{4}T + \tfrac{|\varphi|_{L^{2}([0,T];\R^{m})}^{2}}{\eta}\Big).
      \end{split}
      \end{equation*}
      Thus we complete the proof.
\end{proof}

Based on the above a priori estimate, we can establish the well-posedness of the skeleton equation \eqref{eq:skeletoneq}, as stated by the following lemma.

\begin{lemma}\label{lem:skeletoneqexu}
      Suppose that Assumption \ref{eq:mainass} holds. Then \eqref{eq:skeletoneq} admits a unique solution $\{z^{\varphi}(t)\}_{t \in [-\tau,T]}$.
\end{lemma}

\begin{proof}
      For any $n \in \N$, denote $\sigma_{n} \colon [0,T] \times \R^{d} \times \R^{d} \to \R^{d \times m}$ by
      \begin{equation*}
            \sigma_{n}(t,x,y)
            =
            \begin{cases}
            \sigma(t,x,y),
            & |x|\vee|y| \leq n,
            \\
            \sigma(t,x,y)(2-|x|/n),
            & |y| \leq n < |x| \leq 2n,
            \\
            \sigma(t,x,y)(2-|y|/n),
            & |x| \leq n < |y| \leq 2n,
            \\
            \sigma(t,x,y)(2-|x|/n)(2-|y|/n),
            & n < |x|\wedge|y| \leq |x|\vee |y| \leq 2n,
            \\
            0,
            &\text {otherwise},
            \end{cases}
      \end{equation*}
      where we have used the notations $a \vee b := \max\{a,b\}$ and $a \wedge b := \min\{a,b\}$ for any $a,b \in \R$. It is easy to check that $\sigma_{n}$ is globally Lipschitz continuous with the Lipschitz constant depending on $n$.
      For any fixed $\zeta\in  C_{\phi}([-\tau,T];\R^{d})$, we claim that the following equation
      \begin{equation}
      \begin{split}\label{Gamma_n}
            \frac{\diff{u(t)}}{\diff{t}}
            =&~
            b(t,u(t),u(t-\tau))
            +
            \sigma_n(t,\zeta(t),\zeta(t-\tau))\varphi(t),
            \quad t \in (0,T],
            \\
            u(t)
            =&~
            \phi(t), \quad t \in [-\tau,0]
      \end{split}
      \end{equation}
      admits a unique solution $\Gamma_n(\zeta)\in C_{\phi}([-\tau,T];\R^{d})$. Actually, since $\sigma_{n}$ is bounded and $\varphi \in L^{2}([0,T];\R^{m})$, the integral
      $$
            U_{n}^{\zeta}(t)
            :=
            \int_{0}^{t} \sigma_n(s,\zeta(s),\zeta(s-\tau))\varphi(s)\diff{s},
            \quad t \in [0,T]
      $$
      is well defined. By setting $U_{n}^{\zeta}(t):=0, t\in [-\tau,0]$ and $M_{n}^{\zeta}(t) := u(t)-U_{n}^{\zeta}(t)$, we have
      \begin{equation}\label{Meqqq}
      \begin{split}
            \frac{\diff{M_{n}^{\zeta}(t)}}{\diff{ t}}
            &=
            b(t,M_{n}^{\zeta}(t)
            + U_{n}^{\zeta}(t),M_{n}^{\zeta}(t-\tau)
            + U_{n}^{\zeta}(t-\tau)),
            \quad t\in [0,T],
            \\
            M_{n}^{\zeta}(t)
            &=
            \phi(t), \quad t\in[-\tau,0].
      \end{split}
      \end{equation}
      Denote $B(t,x,y):=b(t,x + U_{n}^{\zeta}(t),y+ U_{n}^{\zeta}(t-\tau)), t \in [0,T],x,y \in \R^{d}$. Then \eqref{Meqqq} can be rewritten as
      \begin{equation}
      \begin{split}\label{Meq}
            \frac{\diff{M_{n}^{\zeta}(t)}}{\diff{ t}}
            &=
            B(t,M_{n}^{\zeta}(t),M_{n}^{\zeta}(t-\tau)),
            \quad t \in [0,T],
            \\
            M_{n}^{\zeta}(t)
            &=
            \phi(t), \quad t\in[-\tau,0].
      \end{split}
      \end{equation}
      Since $B$ is locally Lipschitz continuous, \eqref{Meq} admits a unique local solution. It is verified that
      \begin{equation*}
            \langle x_{1}-x_{2},
            B(t,x_{1},y_{1})-B(t,x_{2},y_{2}) \rangle
            \leq
            K_{2}(|x_{1}-x_{2}|^{2} + |y_{1}-y_{2}|^{2}),
            \quad t \in [0,T],x_{1},x_{2},y_{1},y_{2} \in \R^{d}.
      \end{equation*}
      By the above globally monotone condition and the standard extension arguments, such a local solution of \eqref{Meq} can be extended to the whole interval $[0, T]$.

      For any $\zeta_{1}, \zeta_{2} \in C_{\phi}([-\tau,T];\R^{d})$,
      \begin{align*}
            &\frac{\diff{|\Gamma_n(\zeta_1)(t)-\Gamma_n(\zeta_2)(t)|^2}}{\diff{t}}
            \\=&~
            2\langle \Gamma_n(\zeta_1)(t)-\Gamma_n(\zeta_2)(t),
            b(t,\Gamma_n(\zeta_1)(t),\Gamma_n(\zeta_1)(t-\tau))
            -b(t,\Gamma_n(\zeta_2)(t),\Gamma_n(\zeta_2)(t-\tau))\rangle
            \\&~
            +2\langle \Gamma_n(\zeta_1)(t)-\Gamma_n(\zeta_2)(t),
            (\sigma_n(t,\zeta_1(t),\zeta_1(t-\tau))
            -\sigma_n(t,\zeta_2(t),\zeta_2(t-\tau)))\varphi(t)\rangle .
      \end{align*}
      In view of \eqref{eq:couplemonocond}, we obtain
      \begin{align*}
            &|\Gamma_n(\zeta_1)(t)-\Gamma_n(\zeta_2)(t)|^{2}
            \\=&~
            2\int_0^t\langle
            \Gamma_n(\zeta_1)(s)-\Gamma_n(\zeta_2)(s),
            b(s,\Gamma_n(\zeta_1)(s),\Gamma_n(\zeta_1)(s-\tau))
            -b(s,\Gamma_n(\zeta_2)(s),\Gamma_n(\zeta_2)(s-\tau))
            \rangle \diff{s}
            \\&~+
            2\int_0^t\langle \Gamma_n(\zeta_1)(s)-\Gamma_n(\zeta_2)(s),
            (\sigma_n(s,\zeta_1(s),\zeta_1(s-\tau))
            -\sigma_n(s,\zeta_2(s),\zeta_2(s-\tau)))\varphi(s)
            \rangle \diff{s}
            \\\le&~
            2K_2\int_0^t
            |\Gamma_n(\zeta_1)(s)-\Gamma_n(\zeta_2)(s)|^{2}
            +
            |\Gamma_n(\zeta_1)(s-\tau)-\Gamma_n(\zeta_2)(s-\tau)|^2\diff{s}
            \\&~+
            \int_0^t|\Gamma_n(\zeta_1)(s)-\Gamma_n(\zeta_2)(s)|^2
            +|\sigma_n(s,\zeta_1(s),\zeta_1(s-\tau))
            -\sigma_n(s,\zeta_2(s),\zeta_2(s-\tau))|^2|\varphi(s)|^2\diff{s}
            \\\le&~
            (4K_2+1)\int_0^t
            |\Gamma_n(\zeta_1)(s)-\Gamma_n(\zeta_2)(s)|^2\diff{s}
            +
            C(n)|\varphi|_{L^{2}([0,T];\R^{m})}^{2}
            |\zeta_{1}-\zeta_{2}|_{C_{\phi}([-\tau,T];\R^{d})}^{2}.
      \end{align*}
      The Gronwall inequality yields
      \begin{align*}
            |\Gamma_n(\zeta_1)(t)-\Gamma_n(\zeta_2)(t)|^2
            \le&~
            e^{(4K_2+1)T_0}C(n)
            |\varphi|_{L^{2}([0,T];\R^{m})}^{2}
            |\zeta_{1}-\zeta_{2}|_{C_{\phi}([-\tau,T];\R^{d})}^{2}.
      \end{align*}
      Choosing $T_0 > 0$ such that
      \begin{equation*}
            e^{(4K_2+1)T_0}C(n)
            |\varphi|_{L^{2}([0,T];\R^{m})}^{2}
            <
            1,
      \end{equation*}
      we obtain that $\Gamma_{n}$ is a contraction in $C_{\phi}([-\tau,T_0];\R^{d})$. Hence $\Gamma^{n}$ has a unique fixed point, which is the unique solution of
      \begin{equation}\label{eq:skeletoneq_n}
      \begin{split}
            \frac{\diff{z_n^{\varphi}(t)}}{\diff{t}}
            =&~
            b(t,z_n^{\varphi}(t),z_n^{\varphi}(t-\tau))
            +
            \sigma_n(t,z_n^{\varphi}(t),z_n^{\varphi}(t-\tau))\varphi(t),
            \quad t \in (0,T_0],
            \\
            z_n^{\varphi}(t) =&~ \phi(t), \quad t \in [-\tau,0].
      \end{split}
      \end{equation}
      As $T_{0}$ depends only on $\varphi$, $\phi$, $n$ and $K_2$, we can proceed in the same way in $\left[T_{0}, 2 T_{0}\right]$ and so on. In this way, we get the existence of a unique solution for \eqref{eq:skeletoneq_n} which is defined in the whole interval $[-\tau, T]$. Further, it is verified that
      \begin{equation*}
      \langle x,b(t,x,y) \rangle
      +
      \tfrac{\eta}{2}|\sigma_{n}(t,x,y)|^{2}
      \leq
      K_{4}(1 + |x|^{2} + |y|^{2}),
      \quad t \in [0,T],x,y \in \R^{d}.
      \end{equation*}
      Similar to the proof of Lemma \ref{lem:squarezvarphit}, one has  $|z_{n}^{\varphi}|_{C_{\phi}([-\tau,T];\R^{d})}\le C(K_4,T,\eta,\phi,\varphi)$. This implies $ \sigma_n(t,z_n^{\varphi}(t),z_n^{\varphi}(t-\tau))= \sigma(t,z_n^{\varphi}(t),z_n^{\varphi}(t-\tau))$ for any $t \in [0,T]$ provided that $n\ge C(K_4,T,\eta,\phi,\varphi)$, which together with Lemma \ref{lem:squarezvarphit} ensures the existence of a unique solution for \eqref{eq:skeletoneq}. Thus the proof is complete.
\end{proof}

The following result shows the compactness of an integral operator and will be used to verify the conditions in Lemma \ref{lem:sufficientlemma}.
\begin{lemma}\label{lem:compactG}
      Define the operator $G \colon L^{2}([0,T];\R^{m}) \to C([0,T];\R^{m})$ by
      \begin{equation}\label{eq:mapG}
            G(\varphi)(\cdot)
            =
            \int_{0}^{\cdot} \varphi(s) \diff{s},
            \quad \varphi \in L^{2}([0,T];\R^{m}).
      \end{equation}
      Then $G$ is compact, i.e., for any $\alpha \in (0,\infty)$, $G(S_{\alpha}) := \{G(\varphi) | \varphi \in S_{\alpha}\}$ is relatively compact in $C([0,T];\R^{m})$.
\end{lemma}

\begin{proof}
      Obviously, $G$ is a bounded linear operator. For any given $\alpha \in (0,\infty)$, it suffices to show that $G(S_{\alpha})$ is bounded and equicontinuous in view of the Arz\`{e}la--Ascolil theorem. In fact, for any $\varphi \in S_{\alpha}$, we use the H\"{o}lder inequality to get
      \begin{equation*}
            |G(\varphi)(t)|
            =
            \Big|\int_{0}^{t} \varphi(s) \diff{s}\Big|
            \leq
            \sqrt{T}\Big(\int_{0}^{t} |\varphi(s)|^{2} \diff{s}\Big)^{\frac{1}{2}}
            \leq
            \sqrt{T\alpha},
      \end{equation*}
      which gives the boundedness of $G(S_{\alpha})$. Besides, for any $\kappa > 0$, there exists $\delta := \frac{\kappa^{2}}{\alpha} > 0$ such that for any $t,s \in [0,T]$ with $|t-s| < \delta$, it holds that
      \begin{align*}
            |G(\varphi)(t)-G(\varphi)(s)|
            =
            \Big|\int_{s}^{t} \varphi(r) \diff{r}\Big|
            \leq
            \sqrt{|t-s|}
            \Big(\int_{s}^{t} |\varphi(r)|^{2}
            \diff{r}\Big)^{\frac{1}{2}}
            \leq
            \sqrt{|t-s|\alpha}
            <
            \kappa
      \end{align*}
      for all $\varphi \in S_{\alpha}$, which validates the equicontinuity of $G(S_{\alpha})$. Thus we complete the proof.
\end{proof}

The following lemma validates that the map $F$, given by \eqref{eq:mapF},
satisfies the first condition in Lemma \ref{lem:sufficientlemma}.
\begin{lemma}\label{lem:ratefunction}
      Suppose that Assumption \ref{eq:mainass} holds. Then for each $\alpha \in (0,\infty)$, the set
      \begin{equation*}
            \Big\{F\Big(\int_{0}^{\cdot}
            \varphi(s) \diff{s}\Big) \Big|
            \varphi \in S_{\alpha}\Big\}
      \end{equation*}
      is a compact subset in $C_{\phi}([-\tau,T];\R^{d})$.
\end{lemma}

\begin{proof}
      By \eqref{eq:mapG}, $F \circ G$ is a map from $L^{2}([0,T];\R^{m})$ to $C_{\phi}([-\tau,T];\R^{d})$. We first show that the compound map $F \circ G$ is continuous from $S_{\alpha}$ to $C_{\phi}([-\tau,T];\R^{d})$. To this end, let $\{\varphi_{n}\}_{n \in \N} \subset S_{\alpha}$, $\varphi \in S_{\alpha}$ be such that $\varphi_{n} \to \varphi$ in $S_{\alpha}$ as $n \to \infty$. It follows from \eqref{eq:skeletoneq} that $F \circ G(\varphi) = z^{\varphi}$ and $F \circ G(\varphi^{n}) = z^{\varphi^{n}}$, which immediately shows that $z^{\varphi_{n}}(t)-z^{\varphi}(t) = 0$ for all $t \in [-\tau,0]$ and
      \begin{equation}\label{eq:zvnt.zvt}
      \begin{split}
            \frac{\diff{(z^{\varphi_{n}}(t)-z^{\varphi}(t))}}{\diff{t}}
            =&~
            b(t,z^{\varphi_{n}}(t),z^{\varphi_{n}}(t-\tau))
            -
            b(t,z^{\varphi}(t),z^{\varphi}(t-\tau))
            \\&~+
            \sigma(t,z^{\varphi_{n}}(t),z^{\varphi_{n}}(t-\tau))
            \varphi^{n}(t)
            -
            \sigma(t,z^{\varphi}(t),z^{\varphi}(t-\tau))
            \varphi(t),
            ~ t \in (0,T].
      \end{split}
      \end{equation}
      Applying \eqref{eq:couplemonocond}, the Schwarz inequality and the weighted Young inequality $ab \leq \eta a^{2} + \frac{b^{2}}{4\eta}$ for any $a,b \in \R$ leads to
      \begin{align*}
            &~\frac{1}{2}
            \frac{\diff{|z^{\varphi_{n}}(t)-z^{\varphi}(t)|^{2}}}{\diff{t}}
            =
            \Big\langle
            z^{\varphi_{n}}(t)-z^{\varphi}(t),
            \frac{\diff{(z^{\varphi_{n}}(t)-z^{\varphi}(t))}}{\diff{t}}
            \Big\rangle
            \\=&~
            \big\langle z^{\varphi_{n}}(t)-z^{\varphi}(t),
            b(t,z^{\varphi_{n}}(t),z^{\varphi_{n}}(t-\tau))
            -
            b(t,z^{\varphi}(t),z^{\varphi}(t-\tau)) \big\rangle
            \\&~+
            \big\langle z^{\varphi_{n}}(t)-z^{\varphi}(t),
            \big(\sigma(t,z^{\varphi_{n}}(t),z^{\varphi_{n}}(t-\tau))
            -
            \sigma(t,z^{\varphi}(t),z^{\varphi}(t-\tau))\big)
            \varphi^{n}(t) \big\rangle
            \\&~+
            \big\langle z^{\varphi_{n}}(t)-z^{\varphi}(t),
            \sigma(t,z^{\varphi}(t),z^{\varphi}(t-\tau))
            (\varphi^{n}(t)-\varphi(t)) \big\rangle
          \\\leq&~
            \langle z^{\varphi_{n}}(t)-z^{\varphi}(t),
            b(t,z^{\varphi_{n}}(t),z^{\varphi_{n}}(t-\tau))
            -
            b(t,z^{\varphi}(t),z^{\varphi}(t-\tau)) \rangle
            \\&~+
            \eta|\sigma(t,z^{\varphi_{n}}(t),z^{\varphi_{n}}(t-\tau))
            -
            \sigma(t,z^{\varphi}(t),z^{\varphi}(t-\tau))|^{2}
            +
            \tfrac{1}{4\eta}|z^{\varphi_{n}}(t)-z^{\varphi}(t)|^{2}
            |\varphi^{n}(t)|^{2}
            \\&~+
            \big\langle z^{\varphi_{n}}(t)-z^{\varphi}(t),
            \sigma(t,z^{\varphi}(t),z^{\varphi}(t-\tau))
            (\varphi^{n}(t)-\varphi(t)) \big\rangle
          \\\leq&~
            K_{2}\big(|z^{\varphi_{n}}(t)-z^{\varphi}(t)|^{2}
            + |z^{\varphi_{n}}(t-\tau)-z^{\varphi}(t-\tau)|^{2}\big)
            +
            \tfrac{1}{4\eta}|z^{\varphi_{n}}(t)-z^{\varphi}(t)|^{2}
            |\varphi^{n}(t)|^{2}
            \\&~+
            \big\langle z^{\varphi_{n}}(t)-z^{\varphi}(t),
            \sigma(t,z^{\varphi}(t),z^{\varphi}(t-\tau))
            (\varphi^{n}(t)-\varphi(t)) \big\rangle.
      \end{align*}
      It follows that
      \begin{equation}\label{eq:zvarphinzvarphi}
      \begin{split}
            |z^{\varphi_{n}}(t)-z^{\varphi}(t)|^{2}
            \leq
            \int_{0}^{t}
            \sup_{r \in [-\tau,s]}
            |z^{\varphi_{n}}(r)-z^{\varphi}(r)|^{2}
            \big(4K_{2} + \tfrac{|\varphi^{n}(s)|^{2}}
            {2\eta}\big)\diff{s} + \Xi_{n}(t),
      \end{split}
      \end{equation}
      where
      \begin{align*}
            \Xi_{n}(t)
            :=
            2\int_{0}^{t}\big\langle z^{\varphi_{n}}(s)-z^{\varphi}(s),
            \sigma(s,z^{\varphi}(s),z^{\varphi}(s-\tau))
            (\varphi^{n}(s)-\varphi(s)) \big\rangle \diff{s},
            \quad t \in [0,T].
      \end{align*}

      Next we will make an upper bound estimate for $\Xi_{n}(t)$. For this purpose, we define
      \begin{equation}\label{eq:xint}
            \xi_{n}(t)
            :=
            \int_{0}^{t}
            \sigma(s,z^{\varphi}(s),z^{\varphi}(s-\tau))
            (\varphi_{n}(s)-\varphi(s))\diff{s},
            \quad t \in [0,T].
      \end{equation}
      Applying the integration by parts formula
      and \eqref{eq:zvnt.zvt} yields
      \begin{equation}\label{eq:Xint}
      \begin{split}
            \Xi_{n}(t)
            =&~
            2\big\langle z^{\varphi_{n}}(s)-z^{\varphi}(s),
            \xi_{n}(s) \big\rangle\big|_{0}^{t}
            -
            2\int_{0}^{t}\big\langle
            (z^{\varphi_{n}}(s)-z^{\varphi}(s))^{'},
            \xi_{n}(s) \big\rangle \diff{s}
            \\=&~
            2\big\langle z^{\varphi_{n}}(t)-z^{\varphi}(t),
            \xi_{n}(t) \big\rangle
            \\&~-
            2\int_{0}^{t}\big\langle
            b(s,z^{\varphi_{n}}(s),z^{\varphi_{n}}(s-\tau))
            -
            b(s,z^{\varphi}(s),z^{\varphi}(s-\tau)),
            \xi_{n}(s) \big\rangle \diff{s}
            \\&~-
            2\int_{0}^{t}\big\langle
            \sigma(s,z^{\varphi_{n}}(s),z^{\varphi_{n}}(s-\tau))
            \varphi^{n}(s)
            -
            \sigma(s,z^{\varphi}(s),z^{\varphi}(s-\tau))
            \varphi(s),
            \xi_{n}(s) \big\rangle \diff{s}
            \\=:&~
            \Xi_{n}^{a}(t) + \Xi_{n}^{b}(t) + \Xi_{n}^{c}(t).
      \end{split}
      \end{equation}
      By the weighted Young inequality,
      \begin{align*}
            \Xi_{n}^{a}(t)
            =
            2\big\langle z^{\varphi_{n}}(t)-z^{\varphi}(t),
            \xi_{n}(t) \big\rangle
            \leq
            \tfrac{1}{2}|z^{\varphi_{n}}(t)-z^{\varphi}(t)|^{2}
            +
            2\sup_{s \in [0,t]}|\xi_{n}(s)|^{2}.
      \end{align*}
      For $\Xi_{n}^{b}(t)$ and $\Xi_{n}^{c}(t)$, one can use the Schwarz inequality, the H\"{o}lder inequality, \eqref{eq:bsigmasuperlinear}, Lemma \ref{lem:squarezvarphit} and $\{\varphi_{n}\}_{n \in \N} \subset S_{\alpha}$, $\varphi \in S_{\alpha}$ to obtain
      \begin{align*}
            \Xi_{n}^{b}(t)
            \leq&~
            2\sup_{s \in [0,t]}|\xi_{n}(s)|
            \int_{0}^{t}
            \big|b(s,z^{\varphi_{n}}(s),z^{\varphi_{n}}(s-\tau))\big|
            +
            \big|b(s,z^{\varphi}(s),z^{\varphi}(s-\tau))\big|\diff{s}
            \\\leq&~
            2K_{6}\sup_{s \in [0,t]}|\xi_{n}(s)|
            \int_{0}^{t} \big(2 + |z^{\varphi_{n}}(s)|^{q}
            +
            |z^{\varphi_{n}}(s-\tau)|^{q} + |z^{\varphi}(s)|^{q}
            +
            |z^{\varphi}(s-\tau)|^{q}\big)\diff{s}
            \\\leq&~
            C\sup_{s \in [0,t]}|\xi_{n}(s)|
      \end{align*}
      and
      \begin{align*}
            \Xi_{n}^{c}(t)
            \leq&~
            2\sup_{s \in [0,t]}|\xi_{n}(s)|
            \int_{0}^{t}\big| \sigma(s,z^{\varphi_{n}}(s),
            z^{\varphi_{n}}(s-\tau))\varphi^{n}(s)\big|
            +
            \big|\sigma(s,z^{\varphi}(s),z^{\varphi}(s-\tau))
            \varphi(s)\big| \diff{s}
            \\\leq&~
            2\sup_{s \in [0,t]}|\xi_{n}(s)|
            \Big\{\Big(\int_{0}^{t}| \sigma(s,z^{\varphi_{n}}(s),
            z^{\varphi_{n}}(s-\tau))|^{2}\diff{s}\Big)^{\frac{1}{2}}
            \Big(\int_{0}^{t}|\varphi^{n}(s)|^{2}\diff{s}\Big)^{\frac{1}{2}}
            \\&~+
            \Big(\int_{0}^{t}
            |\sigma(s,z^{\varphi}(s),z^{\varphi}(s-\tau))|^{2} \diff{s}\Big)^{\frac{1}{2}}
            \Big(\int_{0}^{t}|\varphi(s)|^{2} \diff{s}\Big)^{\frac{1}{2}}\Big\}
            \\\leq&~
            2K_{6}\sqrt{\alpha}\sup_{s \in [0,t]}|\xi_{n}(s)|
            \Big\{\Big(\int_{0}^{t}
            1 + |z^{\varphi_{n}}(s)|^{2q}
            + |z^{\varphi_{n}}(s-\tau)|^{2q}
            \diff{s}\Big)^{\frac{1}{2}}
            \\&~+
            \Big(\int_{0}^{t}
            1 + |z^{\varphi}(s)|^{2q}
            + |z^{\varphi}(s-\tau)|^{2q}
            \diff{s}\Big)^{\frac{1}{2}}\Big\}
            \\\leq&~
            C\sup_{s \in [0,t]}|\xi_{n}(s)|.
      \end{align*}
      Substituting the estimates on $\Xi_{n}^{a}(t)$,
      $\Xi_{n}^{b}(t)$ and $\Xi_{n}^{c}(t)$ into \eqref{eq:Xint}, one has
      \begin{equation*}
            \Xi_{n}(t)
            \leq
            \tfrac{1}{2}|z^{\varphi_{n}}(t)-z^{\varphi}(t)|^{2}
            +
            C|\xi_{n}|_{C([0,T];\R^{d})}^{2}
            +
            C|\xi_{n}|_{C([0,T];\R^{d})},
            \quad t \in [0,T],
      \end{equation*}
      which together with \eqref{eq:zvarphinzvarphi} leads to
      \begin{equation*}
      \begin{split}
            \sup_{r \in [-\tau,t]}
            |z^{\varphi_{n}}(r)&~-z^{\varphi}(r)|^{2}
            \leq
            C\big(|\xi_{n}|_{C([0,T];\R^{d})}
            +
            |\xi_{n}|_{C([0,T];\R^{d})}^{2}\big)
            \\&~+
            \int_{0}^{t}
            \sup_{r \in [-\tau,s]}
            |z^{\varphi_{n}}(r)-z^{\varphi}(r)|^{2}
            \big(8K_{2} + |\varphi^{n}(s)|^{2}/\eta\big)\diff{s}.
      \end{split}
      \end{equation*}
      It follows from the Gronwall inequality and $\{\varphi_{n}\}_{n \in \N} \subset S_{\alpha}$ that
      \begin{equation*}
      \begin{split}
            \sup_{r \in [-\tau,t]}
            |z^{\varphi_{n}}(r)-z^{\varphi}(r)|^{2}
            \leq&~
            C\big(|\xi_{n}|_{C([0,T];\R^{d})}
            +
            |\xi_{n}|_{C([0,T];\R^{d})}^{2}\big)
            \exp\Big(\int_{0}^{t}
                8K_{2} + \tfrac{|\varphi^{n}(s)|^{2}}{\eta}
            \diff{s}\Big)
            \\\leq&~
            Ce^{8K_{2}T+\tfrac{\alpha}{\eta}}
            \big(|\xi_{n}|_{C([0,T];\R^{d})}
            +
            |\xi_{n}|_{C([0,T];\R^{d})}^{2}\big),
            \quad t \in [0,T],
      \end{split}
      \end{equation*}
      which implies
      \begin{equation}\label{eq:zvarnzvar2}
            |z^{\varphi_{n}}-z^{\varphi}|_{
            C_{\phi}([-\tau,T];\R^{d})}^{2}
            \leq
            Ce^{8K_{2}T+\tfrac{\alpha}{\eta}}
            \big(|\xi_{n}|_{C([0,T];\R^{d})}
            +
            |\xi_{n}|_{C([0,T];\R^{d})}^{2}\big).
      \end{equation}
      The continuity of $F \circ G$ will be established once we show $\lim\limits_{n \to \infty}|\xi_{n}|_{C([0,T];\R^{d})} = 0$. Indeed, for any $h \in L^{2}([0,T];\R^{d})$, one can use \eqref{eq:bsigmasuperlinear}, Lemma \ref{lem:squarezvarphit} and $\varphi \in S_{\alpha}$ to show
      \begin{equation*}
      \begin{split}
            &~\int_{0}^{T}
                |\sigma(s,z^{\varphi}(s),
                z^{\varphi}(s-\tau))^{*}h(s)|^{2}
            \diff{s}
            \\\leq&~
            C\int_{0}^{T}
                (1 + |z^{\varphi}(s)|^{2q}
                   + |z^{\varphi}(s-\tau)|^{2q})|h(s)|^{2}
            \diff{s}
            \leq
            C|h|_{L^{2}([0,T];\R^{d})}^{2}
            < \infty,
      \end{split}
      \end{equation*}
      which together with $\varphi_{n} \to \varphi$ in $S_{\alpha}$ as $n \to \infty$ yields
      \begin{equation}\label{eq:410410410}
      \begin{split}
            &~\lim_{n \to \infty}
            \big\langle \sigma(\cdot,z^{\varphi}(\cdot),
            z^{\varphi}(\cdot-\tau))(\varphi_{n}-\varphi),
            h \big\rangle_{L^{2}([0,T];\R^{d})}
            \\=&~
            \lim_{n \to \infty}
            \big\langle \varphi_{n}-\varphi,
            \sigma(\cdot,z^{\varphi}(\cdot),z^{\varphi}(\cdot-\tau))^{*}h
            \big\rangle_{L^{2}([0,T];\R^{m})}
            =
            0.
      \end{split}
      \end{equation}
      That is to say, $\sigma(\cdot,z^{\varphi}(\cdot),
      z^{\varphi}(\cdot-\tau))\varphi_{n} \to \sigma(\cdot,z^{\varphi}(\cdot),z^{\varphi}(\cdot-\tau))\varphi$ in $L^{2}([0,T];\R^{d})$ with respect to its weak topology as $n \to \infty$. In combination with Lemma \ref{lem:compactG}, we have $\lim\limits_{n \to \infty}|\xi_{n}|_{C([0,T];\R^{d})} = 0$, which together with \eqref{eq:zvarnzvar2} yields the continuity $F \circ G$. Finally, for each $\alpha \in (0,\infty)$, the compactness of $\{F \circ G(\varphi):\varphi \in S_{\alpha}\}$ follows from the compactness of $S_{\alpha}$ and the continuity of $F \circ G$. Thus the proof is complete.
\end{proof}

We now state the following stochastic Gronwall lemma (see \cite[Theorem 4]{scheutzow2013stochastic}), which is a key tool in obtaining the uniform moment estimate for the stochastic controlled equation \eqref{eq:Yvaruvart}.

\begin{lemma}\label{lem:stochasticgronwall}
      Let $Z$ and $H$ be nonnegative, adapted processes with continuous paths and assume that $\psi$ is nonnegative and progressively measurable. Let $M$ be a continuous local martingale starting at zero. If
      \begin{equation*}
            Z(t)
            \leq
            H(t)
            +
            \int_{0}^{t}
                Z(s)\psi(s)
            \diff{s}
            +
            M(t)
      \end{equation*}
      holds for all $t \geq 0$, then for $\tilde{p} \in (0,1)$, and $\mu,\nu > 1$ such that $\frac{1}{\mu} + \frac{1}{\nu} = 1$ and $\tilde{p}\nu < 1$, we have
      \begin{equation*}
            \E\Big[\sup_{s \in [0,t]}
            Z^{\tilde{p}}(s)\Big]
            \leq
            (c_{\tilde{p}\nu}+1)^{1/\nu}
            \Big\{\E\Big[\exp\Big(\tilde{p}
            \mu\int_{0}^{t}\psi(s)\diff{s}
            \Big)\Big]\Big\}^{1/\mu}
            \Big\{\E\Big[\Big(\sup_{s \in [0,t]}H(s)\Big)^{\tilde{p}\nu}\Big]\Big\}^{1/\nu}
      \end{equation*}
      with $c_{\tilde{p}\nu} :=
      \Big(4 \wedge \frac{1}{\tilde{p}\nu}\Big)
      \frac{\pi \tilde{p}\nu}{\sin(\pi \tilde{p}\nu)}$.
\end{lemma}

Now we are ready to prove the uniform moment estimate for the solution of the stochastic controlled equation \eqref{eq:Yvaruvart}.
\begin{lemma}\label{lem:momentbound}
      Suppose that Assumption \ref{eq:mainass} holds and let $\{u^{\varepsilon}\}_{\varepsilon > 0} \subset \mathcal{A}_{\alpha}$ for some $\alpha \in (0,\infty)$.
      Let $\{Y^{\varepsilon,u^{\varepsilon}}(t)\}_{t \in [-\tau,T]}$ be the unique strong solution of the stochastic controlled equation  \eqref{eq:Yvaruvart} with the controlled process $u^{\varepsilon}$.
      Then for any $p \in [2,\eta+1)$, there exists a constant $C := C(K_{4},T,\tau,p,\alpha) > 0$ such that
      \begin{equation}\label{eq:momentbound}
            \sup_{\varepsilon \in (0,\frac{1}{2})}
            \E\Big[\sup_{s \in [0,T]} |Y^{\varepsilon,
            u^{\varepsilon}}(s)|^{p}\Big]
            \leq
            C.
      \end{equation}
\end{lemma}

\begin{proof}
      Let $\delta \in (0,1)$ be an arbitrarily small constant such that $\eta+1-\delta > 2$. It suffices to show that for any $p \in [2,\eta+1-\delta]$, there exists a constant $C := C(K_{4},T,\tau,\delta,p,\alpha) > 0$ such that
      \begin{equation}\label{eq:momentboundauxi}
            \sup_{\varepsilon \in (0,\frac{1}{2})}
            \E\Big[\sup_{s \in [0,T]} |Y^{\varepsilon,
            u^{\varepsilon}}(s)|^{p}\Big]
            \leq
            C.
      \end{equation}
      In fact, since letting $\bar{\delta}:=\frac{\delta}{[T/\tau]+1}$ yields $\eta+1-\bar{\delta} > \eta+1-\delta > 2$, one can take any $\bar{p} \in [2,\eta+1-\bar{\delta}]$. By the It\^{o} formula, the Schwarz inequality, the weighted Young inequality $ab \leq \kappa a^{2} + \frac{b^{2}}{4\kappa}$ for any $a,b \in \R$ with $\kappa = \frac{\bar{p}+\bar{\delta}-1}{2}(1-\varepsilon) > 0$ and $\varepsilon \in (0,\frac{1}{2})$, we have
      \begin{align*}
            &~|Y^{\varepsilon,u^{\varepsilon}}(t)|^{\bar{p}+\bar{\delta}}
            \\=&~
            |\phi(0)|^{\bar{p}+\bar{\delta}}
            +
            (\bar{p}+\bar{\delta})\int_{0}^{t}
            |Y^{\varepsilon,u^{\varepsilon}}(s)|^{\bar{p}+\bar{\delta}-2}
            \big\langle Y^{\varepsilon,u^{\varepsilon}}(s),
            b(s,Y^{\varepsilon,u^{\varepsilon}}(s),
            Y^{\varepsilon,u^{\varepsilon}}(s-\tau))
            \big\rangle \diff{s}
            \\&~+
            (\bar{p}+\bar{\delta})\int_{0}^{t}
            |Y^{\varepsilon,u^{\varepsilon}}(s)|^{\bar{p}+\bar{\delta}-2}
            \big\langle Y^{\varepsilon,u^{\varepsilon}}(s),
            \sigma(s,Y^{\varepsilon,u^{\varepsilon}}(s),
            Y^{\varepsilon,u^{\varepsilon}}(s-\tau))
            u^{\varepsilon}(s) \big\rangle \diff{s}
            \\&~+
            \tfrac{(\bar{p}+\bar{\delta})\varepsilon}{2}\int_{0}^{t}
            |Y^{\varepsilon,u^{\varepsilon}}(s)|^{\bar{p}+\bar{\delta}-2}
            |\sigma(s,Y^{\varepsilon,u^{\varepsilon}}(s),
            Y^{\varepsilon,u^{\varepsilon}}(s-\tau))|^{2} \diff{s}
            \\&~+
            \tfrac{(\bar{p}+\bar{\delta})(\bar{p}+\bar{\delta}-2)\varepsilon}{2}
            \int_{0}^{t}
            |Y^{\varepsilon,u^{\varepsilon}}(s)|^{\bar{p}+\bar{\delta}-4}
            |Y^{\varepsilon,u^{\varepsilon}}(s)^{*}
            \sigma(s,Y^{\varepsilon,u^{\varepsilon}}(s),
            Y^{\varepsilon,u^{\varepsilon}}(s-\tau))|^{2} \diff{s}
            \\&~+
            (\bar{p}+\bar{\delta})\sqrt{\varepsilon}\int_{0}^{t}
            |Y^{\varepsilon,u^{\varepsilon}}(s)|^{\bar{p}+\bar{\delta}-2}
            \big\langle Y^{\varepsilon,u^{\varepsilon}}(s),
            \sigma(s,Y^{\varepsilon,u^{\varepsilon}}(s),
            Y^{\varepsilon,u^{\varepsilon}}(s-\tau))\diff{W(s)} \big\rangle
            \\\leq&~
            |\phi(0)|^{\bar{p}+\bar{\delta}}
            +
            (\bar{p}+\bar{\delta})\int_{0}^{t}
            |Y^{\varepsilon,u^{\varepsilon}}(s)|^{\bar{p}+\bar{\delta}-2}
            \big(\big\langle Y^{\varepsilon,u^{\varepsilon}}(s),
            b(s,Y^{\varepsilon,u^{\varepsilon}}(s),
            Y^{\varepsilon,u^{\varepsilon}}(s-\tau))\big\rangle
            \\&~+
            \tfrac{(\bar{p}+\bar{\delta}-1)\varepsilon}{2}
            |\sigma(s,Y^{\varepsilon,u^{\varepsilon}}(s),
            Y^{\varepsilon,u^{\varepsilon}}(s-\tau))|^{2}
            \\&~+
            \big\langle Y^{\varepsilon,u^{\varepsilon}}(s),
            \sigma(s,Y^{\varepsilon,u^{\varepsilon}}(s),
            Y^{\varepsilon,u^{\varepsilon}}(s-\tau))
            u^{\varepsilon}(s) \big\rangle \big)\diff{s}
            \\&~+
            (\bar{p}+\bar{\delta})\sqrt{\varepsilon}\int_{0}^{t}
            |Y^{\varepsilon,u^{\varepsilon}}(s)|^{\bar{p}+\bar{\delta}-2}
            \big\langle Y^{\varepsilon,u^{\varepsilon}}(s),
            \sigma(s,Y^{\varepsilon,u^{\varepsilon}}(s),
            Y^{\varepsilon,u^{\varepsilon}}(s-\tau))\diff{W(s)} \big\rangle
            \\\leq&~
            |\phi(0)|^{\bar{p}+\bar{\delta}}
            +
            (\bar{p}+\bar{\delta})\int_{0}^{t}
            |Y^{\varepsilon,u^{\varepsilon}}(s)|^{\bar{p}+\bar{\delta}-2}
            \big(\big\langle Y^{\varepsilon,u^{\varepsilon}}(s),
            b(s,Y^{\varepsilon,u^{\varepsilon}}(s),
            Y^{\varepsilon,u^{\varepsilon}}(s-\tau))\big\rangle
            \\&~+
            \tfrac{\bar{p}+\bar{\delta}-1}{2}
            |\sigma(s,Y^{\varepsilon,u^{\varepsilon}}(s),
            Y^{\varepsilon,u^{\varepsilon}}(s-\tau))|^{2}\big)\diff{s}
            +
            \tfrac{\bar{p}+\bar{\delta}}
            {2(\bar{p}+\bar{\delta}-1)(1-\varepsilon)}
            \int_{0}^{t}|Y^{\varepsilon,
            u^{\varepsilon}}(s)|^{\bar{p}+\bar{\delta}}
            |u^{\varepsilon}(s)|^{2}\diff{s}
            \\&~+
            (\bar{p}+\bar{\delta})\sqrt{\varepsilon}\int_{0}^{t}
            |Y^{\varepsilon,u^{\varepsilon}}(s)|^{\bar{p}+\bar{\delta}-2}
            \big\langle Y^{\varepsilon,u^{\varepsilon}}(s),
            \sigma(s,Y^{\varepsilon,u^{\varepsilon}}(s),
            Y^{\varepsilon,u^{\varepsilon}}(s-\tau))\diff{W(s)} \big\rangle.
      \end{align*}
      Noting $\bar{p} \in [2,\eta+1-\bar{\delta}]$ and applying \eqref{eq:globalcoercivitycond} as well as the weighted Young inequality yield
      \begin{align*}
            &~|Y^{\varepsilon,u^{\varepsilon}}(t)|^{\bar{p}+\bar{\delta}}
      \\\leq&~
            |\phi(0)|^{\bar{p}+\bar{\delta}}
            +
            \tfrac{1}{1-\varepsilon}\int_{0}^{t}
            |Y^{\varepsilon,u^{\varepsilon}}(s)|^{\bar{p}+\bar{\delta}}
            |u^{\varepsilon}(s)|^{2} \diff{s}
        \\&~+
            (\bar{p}+\bar{\delta})K_{4}\int_{0}^{t}
            |Y^{\varepsilon,u^{\varepsilon}}(s)|^{\bar{p}+\bar{\delta}-2}
            \big(1 + |Y^{\varepsilon,u^{\varepsilon}}(s)|^{2}
            +
            |Y^{\varepsilon,u^{\varepsilon}}(s-\tau)|^{2}\big)\diff{s}
        \\&~+
            (\bar{p}+\bar{\delta})\sqrt{\varepsilon}\int_{0}^{t}
            |Y^{\varepsilon,u^{\varepsilon}}(s)|^{\bar{p}+\bar{\delta}-2}
            \big\langle Y^{\varepsilon,u^{\varepsilon}}(s),
            \sigma(s,Y^{\varepsilon,u^{\varepsilon}}(s),
            Y^{\varepsilon,u^{\varepsilon}}(s-\tau))\diff{W(s)} \big\rangle
      \\\leq&~
            |\phi(0)|^{\bar{p}+\bar{\delta}}
            +
            \tfrac{1}{1-\varepsilon}\int_{0}^{t}
            |Y^{\varepsilon,u^{\varepsilon}}(s)|^{\bar{p}+\bar{\delta}}
            |u^{\varepsilon}(s)|^{2} \diff{s}
            +
            (\bar{p}+\bar{\delta})K_{4}\int_{0}^{t}
            \big(\tfrac{2}{\bar{p}+\bar{\delta}}
            +
            \tfrac{\bar{p}+\bar{\delta}-2}{\bar{p}+\bar{\delta}}
            |Y^{\varepsilon,u^{\varepsilon}}(s)|^{\bar{p}+\bar{\delta}}
            \\&+
            |Y^{\varepsilon,u^{\varepsilon}}(s)|^{\bar{p}+\bar{\delta}}
            +
            \tfrac{\bar{p}+\bar{\delta}-2}{\bar{p}+\bar{\delta}}
            |Y^{\varepsilon,u^{\varepsilon}}(s)|^{\bar{p}+\bar{\delta}}
            +
            \tfrac{2}{\bar{p}+\bar{\delta}}
            |Y^{\varepsilon,u^{\varepsilon}}
            (s-\tau)|^{\bar{p}+\bar{\delta}}\big)\diff{s}
        \\&~+
            (\bar{p}+\bar{\delta})\sqrt{\varepsilon}\int_{0}^{t}
            |Y^{\varepsilon,u^{\varepsilon}}(s)|^{\bar{p}+\bar{\delta}-2}
            \big\langle Y^{\varepsilon,u^{\varepsilon}}(s),
            \sigma(s,Y^{\varepsilon,u^{\varepsilon}}(s),
            Y^{\varepsilon,u^{\varepsilon}}(s-\tau))\diff{W(s)}
            \big\rangle
      \\\leq&~
            |\phi(0)|^{\bar{p}+\bar{\delta}}
            +
            2K_{4}\int_{0}^{t}
            (1 + |Y^{\varepsilon,u^{\varepsilon}}
            (s-\tau)|^{\bar{p}+\bar{\delta}})\diff{s}
            +
            \int_{0}^{t}
            |Y^{\varepsilon,u^{\varepsilon}}(s)|^{\bar{p}+\bar{\delta}}
            \big(3(\bar{p}+\bar{\delta})K_{4}
            +
            \tfrac{|u^{\varepsilon}(s)|^{2}}{1-\varepsilon}\big)\diff{s}
        \\&~+
            (\bar{p}+\bar{\delta})\sqrt{\varepsilon}
            \int_{0}^{t}
                |Y^{\varepsilon,u^{\varepsilon}}(s)|^
                {\bar{p}+\bar{\delta}-2}
            \big\langle Y^{\varepsilon,u^{\varepsilon}}(s),
            \sigma(s,Y^{\varepsilon,u^{\varepsilon}}(s),
            Y^{\varepsilon,u^{\varepsilon}}
            (s-\tau))\diff{W(s)} \big\rangle.
      \end{align*}
      Lemma \ref{lem:stochasticgronwall} with $\tilde{p} = \bar{p}/(\bar{p}+\bar{\delta}) \in (0,1)$, $\nu \in (1,\frac{1}{\tilde{p}})$, $\mu = \frac{\nu}{\nu-1} > 1$ and $\varepsilon \in (0,\frac{1}{2})$ shows
      \begin{align*}
            \E\Big[\sup_{s \in [0,t]}
            |Y^{\varepsilon,u^{\varepsilon}}(s)|^{\bar{p}}\Big]
            \leq&~
            (c_{\tilde{p}\nu}+1)^{1/\nu}
            \Big\{\E\Big[\exp\Big(\tilde{p}\mu
            \int_{0}^{t}
                \big(3(\bar{p}+\bar{\delta})K_{4}
                +
                \tfrac{|u^{\varepsilon}(s)|^{2}}
                {1-\varepsilon}\big)
            \diff{s}\Big)\Big]\Big\}^{1/\mu}
            \\&~\times
            \Big\{\E\Big[\Big(\sup_{s \in [0,t]}
            \Big(|\phi(0)|^{\bar{p}+\bar{\delta}}
            +
            2K_{4}\int_{0}^{s}(1 +
            |Y^{\varepsilon,u^{\varepsilon}}(r-\tau)|
            ^{\bar{p}+\bar{\delta}})
            \diff{r}\Big)\Big)^{\tilde{p}\nu}\Big]\Big\}^{1/\nu}
      \\\leq&~
            (c_{\tilde{p}\nu}+1)^{1/\nu}
            \exp\big(\tilde{p}\big(3(\bar{p}
            + \bar{\delta})K_{4}T
            + 2\alpha\big)\big)
            \Big\{(|\phi(0)|^{\bar{p}+\bar{\delta}}
            + 2K_{4}T)^{\tilde{p}\nu}
        \\&~+
            \E\Big[\Big(2K_{4}
            \int_{0}^{t}
               |Y^{\varepsilon,u^{\varepsilon}}
               (r-\tau)|^{\bar{p}+\bar{\delta}}
            \diff{r}\Big)^{\tilde{p}\nu}\Big]\Big\}^{1/\nu}
      \\\leq&~
            (c_{\tilde{p}\nu}+1)^{1/\nu}
            \exp\big(\tilde{p}\big(3(\bar{p}+\bar{\delta})K_{4}T
            + 2\alpha\big)\big)
            \Big\{(|\phi(0)|^{\bar{p}+\bar{\delta}} + 2K_{4}T)^{\tilde{p}\nu}
            \\&~+ \Big(2K_{4}\int_{0}^{t}
            \E[|Y^{\varepsilon,u^{\varepsilon}}
            (r-\tau)|^{\bar{p}+\bar{\delta}}]
            \diff{r}\Big)^{\tilde{p}\nu}\Big\}^{1/\nu}
      \\\leq&~
            (c_{\tilde{p}\nu}+1)^{1/\nu}
            \exp\big(\tilde{p}
            \big(3(\bar{p}+\bar{\delta})K_{4}T
            + 2\alpha\big)\big)
            \Big\{(|\phi(0)|^{\bar{p} + \bar{\delta}}
            + 2K_{4}T)^{\tilde{p}}
            \\&~+
            \Big(2K_{4}\int_{0}^{t}
            \E[|Y^{\varepsilon,u^{\varepsilon}}
            (r-\tau)|^{\bar{p}+\bar{\delta}}]
            \diff{r}\Big)^{\tilde{p}}\Big\}.
      \end{align*}
      where we have used $\{u^{\varepsilon}\}_{\varepsilon > 0} \subset \mathcal{A}_{\alpha}$, the H\"{o}lder inequality and the Jensen inequality. Hence, for any $\bar{p} \in [2,\eta+1-\bar{\delta}]$, there exists $C := C(K_{4},T,\tau,\delta,\bar{p},\alpha) > 0$ such that
      \begin{equation}\label{eq:317317}
            \E\Big[\sup_{s \in [0,t]}
            |Y^{\varepsilon,u^{\varepsilon}}(s)|^{\bar{p}}\Big]
            \leq
            C\Big(1 + \Big(\int_{0}^{t}
            \E[|Y^{\varepsilon,u^{\varepsilon}}
            (r-\tau)|^{\bar{p}+\bar{\delta}}]\diff{r}\Big)^
            {\bar{p}/(\bar{p}+\bar{\delta})}\Big),
            \quad t \in [0,T].
      \end{equation}
      Now we are in a position to show \eqref{eq:momentboundauxi}. For any given $p \in [2,\eta+1-\delta]$, denote
      \begin{equation*}
            \bar{p}_{i}
            :=
            p + ([T/\tau]+1-i)\bar{\delta},
            \quad i = 1,2,\cdots,[T/\tau]+1
      \end{equation*}
      with $[T/\tau]$ being the integer part of $T/\tau$. Then we have $\bar{p}_{[T/\tau]+1} = p$, $\bar{p}_{i} \in [2,\eta+1-\bar{\delta}]$ for all $i = 1,2,\cdots,[T/\tau]$.
      Moreover, it holds that $\bar{p}_{i+1} + \bar{\delta} = \bar{p}_{i}$ for each $i = 1,2,\cdots,[T/\tau]-1$. For each $i = 1,2,\cdots,[T/\tau]+1$, it follows from \eqref{eq:317317} and  $\bar{p}_{i} \in [2,\eta+1-\bar{\delta}], i = 1,2,\cdots,[T/\tau]$ that
      \begin{equation}\label{eq:318318}
            \E\Big[\sup_{s \in [0,t]}
            |Y^{\varepsilon,u^{\varepsilon}}(s)|^{\bar{p}_{i}}\Big]
            \leq
            C\Big(1 + \Big(\int_{0}^{t}
            \E[|Y^{\varepsilon,u^{\varepsilon}}(r-\tau)|
            ^{\bar{p}_{i}+\bar{\delta}}]\diff{r}\Big)
            ^{\bar{p}_{i}/(\bar{p}_{i}+\bar{\delta})}\Big),
            \quad t \in [0,T].
      \end{equation}
      When $t \in [0,\tau]$, we use $\E[|Y^{\varepsilon,u^{\varepsilon}}(r-\tau)|^{\bar{p}_{1}+\bar{\delta}}] = |\phi(r-\tau)|^{\bar{p}_{1}+\bar{\delta}}
      \leq |\phi|_{C([-\tau,0];\R^{m})}^{\bar{p}_{1}+\bar{\delta}}$ for all $r \in [0,t]$ and \eqref{eq:318318} to obtain
      \begin{equation*}
            \E\Big[\sup_{s \in [0,t]}
            |Y^{\varepsilon,u^{\varepsilon}}(s)|^{\bar{p}_{1}}\Big]
            \leq
            C\big(1 +
            T^{\bar{p}_{1}/(\bar{p}_{1}+\bar{\delta})}|\phi|_{C([-\tau,0];
            \R^{m})}^{\bar{p}_{1}}\big)
      \end{equation*}
      and thus
      \begin{equation*}
      \begin{split}
            \E\Big[\sup_{s \in [-\tau,\tau]}
            |Y^{\varepsilon,u^{\varepsilon}}(s)|^{\bar{p}_{1}}\Big]
            \leq&~
            \E\Big[\sup_{s \in [-\tau,0]}
            |Y^{\varepsilon,u^{\varepsilon}}(s)|^{\bar{p}_{1}}\Big]
            +
            \E\Big[\sup_{s \in [0,\tau]}
            |Y^{\varepsilon,u^{\varepsilon}}(s)|^{\bar{p}_{1}}\Big]
            \\\leq&~
            |\phi|_{C([-\tau,0];\R^{m})}^{\bar{p}_{1}}
            +
            C\big(1 +
            T^{\bar{p}_{1}/(\bar{p}_{1}+\bar{\delta})}|\phi|_{C([-\tau,0];
            \R^{m})}^{\bar{p}_{1}}\big)
            =:C_{\bar{p}_{1}}.
      \end{split}
      \end{equation*}
      For $t \in [0,2\tau]$, applying \eqref{eq:318318}, $\bar{p}_{2}+\bar{\delta} = \bar{p}_{1}$ and the H\"{o}lder inequality gives
      \begin{equation*}
      \begin{split}
            \E\Big[\sup_{s \in [0,t]}
            |Y^{\varepsilon,u^{\varepsilon}}(s)|
            ^{\bar{p}_{2}}\Big]
            \leq&~
            C\Big(1 + \Big(\int_{0}^{t}
            \E[|Y^{\varepsilon,u^{\varepsilon}}(r-\tau)|
            ^{\bar{p}_{2}+\bar{\delta}}]\diff{r}\Big)
            ^{\bar{p}_{2}/(\bar{p}_{2}+\bar{\delta})}\Big)
      \\=&~
            C\Big(1 + \Big(\int_{0}^{t}
            \E[|Y^{\varepsilon,u^{\varepsilon}}(r-\tau)|^{\bar{p}_{1}}]
            \diff{r}\Big)^{\bar{p}_{2}/(\bar{p}_{2}+\bar{\delta})}\Big)
      \\\leq&~
            C\big(1 +
            (TC_{\bar{p}_{1}})
            ^{\bar{p}_{2}/(\bar{p}_{2}+\bar{\delta})}\big)
      \end{split}
      \end{equation*}
      and therefore
      \begin{equation*}
      \begin{split}
            \E\Big[\sup_{s \in [-\tau,2\tau]}
            |Y^{\varepsilon,u^{\varepsilon}}(s)|^{\bar{p}_{2}}\Big]
            \leq&~
            \E\Big[\sup_{s \in [-\tau,0]}
            |Y^{\varepsilon,u^{\varepsilon}}(s)|^{\bar{p}_{2}}\Big]
            +
            \E\Big[\sup_{s \in [0,2\tau]}
            |Y^{\varepsilon,u^{\varepsilon}}(s)|^{\bar{p}_{2}}\Big]
      \\\leq&~
            |\phi|_{C([-\tau,0];\R^{m})}^{\bar{p}_{2}}
            +
            C\big(1 +
            (TC_{\bar{p}_{1}})^{\bar{p}_{2}/
            (\bar{p}_{2}+\bar{\delta})}\big)
            =:C_{\bar{p}_{2}}.
      \end{split}
      \end{equation*}
      Repeating the above procedure, there exists a constant $C_{\bar{p}_{[T/\tau]}} > 0$ such that
      \begin{equation*}
            \E\Big[\sup_{s \in [-\tau,[T/\tau]\tau]}
            |Y^{\varepsilon,u^{\varepsilon}}(s)|
            ^{\bar{p}_{[T/\tau]}}\Big]
            \leq
            C_{\bar{p}_{[T/\tau]}},
      \end{equation*}
      which together with \eqref{eq:318318} yields the required result \eqref{eq:momentboundauxi}. Thus the proof is complete.
\end{proof}

To proceed, we need the following Kolmogorov criterion for the weak relative compactness; see, e.g.,
\cite[Theorem 21.42]{klenke2008probability}.

\begin{lemma}
\label{lem:kolmogorov}
      Let $(X^i,i\in I)$ be a sequence of continuous stochastic process from $\Omega \times [0,+\infty)$ to $\R^d$. Assume that the following conditions are satisfied.
      \begin{itemize}
		    \item [(a)] The family $(\P\circ (X^i(0))^{-1},\,i\in I)$
                      of initial distributions  is tight.
            \item [(b)] There are numbers $C,\beta,\gamma>0$ such that
                      for all $s,t\in[0,\infty)$ and every $i\in I$,
		              $$\E[|X^i(t)-X^i(s)|^{\beta}]
                        \leq
                        C |t-s|^{\gamma+1}.$$
      \end{itemize}
      Then $(\P\circ (X^i)^{-1},\,i\in I)$ is weakly relatively compact in $\mathcal{P}(C([0,\infty);\R^{d}))$, where $\mathcal{P}(C([0,\infty);\R^{d}))$ is the space of all probability measures on $(C([0,\infty);\R^{d}),\mathcal{B}(C([0,\infty);\R^{d})))$.
\end{lemma}

The forthcoming result indicates that the maps $F$ and $F^{\varepsilon}$, described by \eqref{eq:mapF} and \eqref{eq:mapFvarepsilon}, satisfy the second condition in Lemma \ref{lem:sufficientlemma}.
\begin{lemma}\label{lem:ldpresult}
      Suppose that Assumption \ref{eq:mainass} holds with $\eta > 2q-1$.
      If $u \in \mathcal{A}_{\alpha}$, $\{u^{\varepsilon}\}_{\varepsilon > 0} \subset \mathcal{A}_{\alpha}$ for some $\alpha \in (0,\infty)$ satisfies $u^{\varepsilon} \xrightarrow[\varepsilon \to 0]{d} u$ as $S_{\alpha}$-valued random variables, then
      \begin{equation}\label{eq:FvarF}
            F^{\varepsilon}\Big(W + \frac{1}{\sqrt{\varepsilon}} \int_{0}^{\cdot} u^{\varepsilon}(s) \diff{s}\Big)
            \xrightarrow[\varepsilon \to 0]{d}
            F\Big(\int_{0}^{\cdot} u(s) \diff{s}\Big).
      \end{equation}
\end{lemma}

\begin{proof}
      For a given $\alpha \in (0,\infty)$, let $u \in \mathcal{A}_{\alpha}$, $\{u^{\varepsilon}\}_{\varepsilon > 0} \subset \mathcal{A}_{\alpha}$ satisfy $u^{\varepsilon} \xrightarrow[\varepsilon \to 0]{d} u$. Notice that
      $F(\int_{0}^{\cdot} u(s) \diff{s})$ is the solution of
      \begin{equation}
      \begin{split}
            \diff{z^{u}(t)}
            =&~
            b(t,z^{u}(t),z^{u}(t-\tau))\diff{t}
            +
            \sigma(t,z^{u}(t),z^{u}(t-\tau))u(t)\diff{t},
            \quad t \in (0,T],
            \\
            z^{u}(t) =&~ \phi(t), \quad t \in [-\tau,0].
      \end{split}
      \end{equation}
      It is equivalent to show $Y^{\varepsilon,u^{\varepsilon}} \xrightarrow[\varepsilon \to 0]{d} z^{u}$ in view of \eqref{eq:Yvaruvartttt}.

      For this purpose, we first verify that $\{(Y^{\varepsilon,u^{\varepsilon}},u^{\varepsilon})\}_{\varepsilon > 0}$ is tight as a family of random variables taking values in the Polish space $C([-\tau,T];\R^{d}) \times S_{\alpha}$. Since $S_{\alpha}$ is compact, we know that $\{u^{\varepsilon}\}_{\varepsilon > 0}$ is tight as a family of $S_{\alpha}$-valued random variables. Since $S_{\alpha}$ and $C([-\tau,T];\R^{d})$ are Polish spaces, we need to validate that $\{Y^{\varepsilon,u^{\varepsilon}}\}_{\varepsilon > 0}$ is tight as a family of $C([-\tau,T];\R^{d})$-valued random variables. In fact, $Y^{\varepsilon,u^{\varepsilon}}(0) = \phi(0)$ automatically yields the tightness of the initial distributions $\{\P \circ (Y^{\varepsilon,u^{\varepsilon}}(0))^{-1}\}_{\varepsilon > 0}$. To show the second condition of Lemma \ref{lem:kolmogorov}, we use \eqref{eq:Yvaruvart}, the H\"{o}lder inequality and
      the Burkholder--Davis--Gundy inequality to obtain that for any $\beta \in (2,\frac{\eta+1}{q})$ and $s,t \in [0,T]$ with $s \leq t$,
      \begin{align*}
            \E\big[|Y^{\varepsilon,u^{\varepsilon}}(t)
            &~-Y^{\varepsilon,u^{\varepsilon}}(s)|^{\beta}\big]
      \leq
            3^{\beta-1}\E\Big[\Big|\int_{s}^{t}
            b(r,Y^{\varepsilon,u^{\varepsilon}}(r),
            Y^{\varepsilon,u^{\varepsilon}}(r-\tau))
            \diff{r}\Big|^{\beta}\Big]
            \\&~+
            3^{\beta-1}\E\Big[\Big|\int_{s}^{t}
            \sigma(r,Y^{\varepsilon,u^{\varepsilon}}(r),
            Y^{\varepsilon,u^{\varepsilon}}(r-\tau))
            u^{\varepsilon}(r)
            \diff{r}\Big|^{\beta}\Big]
            \\&~+
            3^{\beta-1}\E\Big[\Big|\int_{s}^{t}
            \sqrt{\varepsilon}
            \sigma(r,Y^{\varepsilon,u^{\varepsilon}}(r),
            Y^{\varepsilon,u^{\varepsilon}}(r-\tau))
            \diff{W(r)}\Big|^{\beta}\Big]
      \\\leq&~
            (3(t-s))^{\beta-1}\int_{s}^{t}
            \E\big[|b(r,Y^{\varepsilon,u^{\varepsilon}}(r),
            Y^{\varepsilon,u^{\varepsilon}}(r-\tau))
            |^{\beta}\big]\diff{r}
            \\&~+
            3^{\beta-1}\E\Big[\Big(\int_{s}^{t}
            |\sigma(r,Y^{\varepsilon,u^{\varepsilon}}(r),
            Y^{\varepsilon,u^{\varepsilon}}(r-\tau))|^{2}
            \diff{r}\Big)^{\frac{\beta}{2}}
            \Big(\int_{s}^{t}
            |u^{\varepsilon}(r)|^{2}
            \diff{r}\Big)^{\frac{\beta}{2}}\Big]
            \\&~+
            3^{\beta-1}
            \Big(\frac{\beta(\beta-1)}{2}\Big)^{\frac{\beta}{2}}
            (t-s)^{\frac{\beta-2}{2}}
            \int_{s}^{t}\E\big[|\sqrt{\varepsilon}
            \sigma(r,Y^{\varepsilon,u^{\varepsilon}}(r),
            Y^{\varepsilon,u^{\varepsilon}}(r-\tau))
            |^{\beta}\big]\diff{r}.
      \end{align*}
      Applying the H\"{o}lder inequality, \eqref{eq:bsigmasuperlinear} and Lemma \ref{lem:momentbound} leads to
      \begin{align*}
            \E\big[|Y^{\varepsilon,u^{\varepsilon}}(t)
            &~-Y^{\varepsilon,u^{\varepsilon}}(s)|^{\beta}\big]
            \leq
            C(t-s)^{\beta-1}\int_{s}^{t}
            \E\big[|b(r,Y^{\varepsilon,u^{\varepsilon}}(r),
            Y^{\varepsilon,u^{\varepsilon}}(r-\tau))|
            ^{\beta}\big]\diff{r}
            \\&~+
            C(1+\varepsilon^{\frac{\beta}{2}})(t-s)
            ^{\frac{\beta-2}{2}}\int_{s}^{t}
            \E\big[|\sigma(r,Y^{\varepsilon,u^{\varepsilon}}(r),
            Y^{\varepsilon,u^{\varepsilon}}(r-\tau))|
            ^{\beta}\big]\diff{r}
      \\\leq&~
            C(t-s)^{\beta-1}\int_{s}^{t}
            \big(1
            +
            \E[|Y^{\varepsilon,u^{\varepsilon}}(r)|^{q\beta}]
            +
            \E[|Y^{\varepsilon,u^{\varepsilon}}(r-\tau))|
            ^{q\beta}]\big)\diff{r}
            \\&~+
            C(t-s)^{\frac{\beta-2}{2}}\int_{s}^{t}
            \big(1
            +
            \E[|Y^{\varepsilon,u^{\varepsilon}}(r)|^{q\beta}]
            +
            \E[|Y^{\varepsilon,u^{\varepsilon}}
            (r-\tau))|^{q\beta}]\big)
            \diff{r}
      \\\leq&~
            C(t-s)^{\frac{\beta}{2}}
            =
            C(t-s)^{(\frac{\beta}{2}-1)+1},
            \quad
            \beta \in (2,\tfrac{\eta+1}{q}),
      \end{align*}
      which together with Lemma \ref{lem:kolmogorov} shows that $\{Y^{\varepsilon,u^{\varepsilon}}\}_{\varepsilon > 0}$ is tight as a family of $C([-\tau,T];\R^{d})$-valued random variables. Thus $\{(Y^{\varepsilon,u^{\varepsilon}},u^{\varepsilon})\}_{\varepsilon > 0}$ is tight as a family of random variables taking values in the Polish space $C([-\tau,T];\R^{d}) \times S_{\alpha}$. It follows from the Prokhorov theorem (see, e.g., \cite[Theorem A.3.15]{dupuis1997weak}) that $\{(Y^{\varepsilon,u^{\varepsilon}},u^{\varepsilon})\}_{\varepsilon > 0}$ is weakly relatively compact in $C([-\tau,T];\R^{d}) \times S_{\alpha}$, which means that there exists a subsequence $\varepsilon_{n} \to 0$ (as $n \to \infty$) such that $\{(Y^{\varepsilon_{n},u^{\varepsilon_{n}}}, u^{\varepsilon_{n}})\}_{\varepsilon_{n} > 0}$ converges in distribution to an element with values in $C([-\tau,T];\R^{d}) \times S_{\alpha}$. According to the Skorohod representation theorem (see, e.g., \cite[Theorem A.3.9]{dupuis1997weak}), there exists a probability space $(\widetilde{\Omega},\widetilde{\F},\widetilde{\P})$ on which a $C([-\tau,T];\R^{d}) \times S_{\alpha}$-valued random variable $(\widetilde{X},\widetilde{u})$ is such that $\{(Y^{\varepsilon_{n},u^{\varepsilon_{n}}}, u^{\varepsilon_{n}})\}_{\varepsilon_{n} > 0}$ converges in distribution to $(\widetilde{X},\widetilde{u})$.

      Denote by $\E_{\widetilde{\P}}$ the expectation with respect to the probability measure $\widetilde{\P}$. Next we show that $\widetilde{X}$ satisfies
      \begin{equation}\label{eq:widetildeXt}
            \widetilde{X}(t) = \phi(0)
            +
            \int_{0}^{t}
                b(s,\widetilde{X}(s),\widetilde{X}(s-\tau))
            \diff{s}
            +
            \int_{0}^{t}
                \sigma(s,\widetilde{X}(s),\widetilde{X}(s-\tau))
                \widetilde{u}(s)
            \diff{s},
      \quad t \in [0,T],\widetilde{\P}\text{-a.s.}
      \end{equation}
      To this end, for each given $t \in [0,T]$, we define the map $\Psi_{t} \colon C([-\tau,T];\R^{d}) \times S_{\alpha} \to [0,1]$ by
      $$\Psi_{t}(f,\varphi) :=
      1 \wedge \Big|f(t)-\phi(0)- \int_{0}^{t} b(s,f(s),f(s-\tau))\diff{s}
      - \int_{0}^{t} \sigma(s,f(s),f(s-\tau))\varphi(s)\diff{s}\Big|.$$
      We claim that $\Psi_{t}$ is continuous. Actually,
      let $f_{n} \to f$ in $C([-\tau,T];\R^{d})$ and $\varphi_{n} \to \varphi$ in $S_{\alpha}$ with respect to the weak topology of $L^{2}([0,T];\R^{m})$. Noting that the elementary inequality
      $\big|1 \wedge |a| - 1 \wedge |b|\big| \leq 1 \wedge |a-b| \leq |a-b|$ holds for all $a,b \in \R$, we use the H\"{o}lder inequality, $\{\varphi_{n}\}_{n \in \N} \subset S_{\alpha}$,
      \eqref{eq:bpolynomialgrow} and \eqref{eq:sigmapolynomialgrow} to obtain
      \begin{align*}
            |\Psi_{t}(f_{n},\varphi_{n})
            &~-\Psi_{t}(f,\varphi)|
      \leq
            |f_{n}(t)-f(t)|
            +
            \int_{0}^{t} |b(s,f_{n}(s),f_{n}(s-\tau))
            -b(s,f(s),f(s-\tau))|\diff{s}
            \\&~+
            \Big(\int_{0}^{t}
            \big|\sigma(s,f_{n}(s),f_{n}(s-\tau)) -\sigma(s,f(s),f(s-\tau))\big|^{2}
            \diff{s}\Big)^{\frac{1}{2}}
            \Big(\int_{0}^{t} |\varphi_{n}(s)|^{2}
            \diff{s}\Big)^{\frac{1}{2}}
            \\&~+
            \Big|\int_{0}^{t}\sigma(s,f(s),f(s-\tau))
            (\varphi_{n}(s)-\varphi(s))\diff{s}\Big|
      \\\leq&~
            |f_{n}-f|_{C([-\tau,T];\R^{d})}
            +
            C\int_{0}^{t}
                 \big(|f_{n}(s)-f(s)|
                 + |f_{n}(s-\tau)-f(s-\tau)|\big)
                 \\&~\times
                 \big(1 + |f_{n}(s)|^{q-1}
                  + |f_{n}(s-\tau)|^{q-1}
                  + |f(s)|^{q-1}
                  + |f(s-\tau)|^{q-1}\big)
            \diff{s}
            \\&~+
            \sqrt{\alpha}C
            \Big(\int_{0}^{t}
                     \big(|f_{n}(s)-f(s)|
                     + |f_{n}(s-\tau)-f(s-\tau)|\big)^{2}
                 \\&~\times
                     \big(1 + |f_{n}(s)|^{q-1}
                     + |f_{n}(s-\tau)|^{q-1}
                     + |f(s)|^{q-1}
                     + |f(s-\tau)|^{q-1}\big)^{2}
                 \diff{s}\Big)^{\frac{1}{2}}
            \\&~+
            \Big|\int_{0}^{t}
                     \sigma(s,f(s),f(s-\tau))
                     (\varphi_{n}(s)-\varphi(s))
                 \diff{s}\Big|
      \\\leq&~
            |f_{n}-f|_{C([-\tau,T];\R^{d})}
            +
            C(T+\sqrt{T\alpha})|f_{n}-f|_{C([-\tau,T];\R^{d})}
            \\&~+
            \Big|\int_{0}^{t}\sigma(s,f(s),f(s-\tau))
            (\varphi_{n}(s)-\varphi(s))\diff{s}\Big|,
      \end{align*}
      where we have used the boundedness of $\{|f_{n}|_{C([-\tau,T];\R^{d})}\}_{n \in \N}$. Similar to the proof of \eqref{eq:410410410}, one can use $\varphi_{n} \to \varphi$ in $S_{\alpha}$ to show $\sigma(\cdot,f(\cdot),f(\cdot-\tau))\varphi_{n}(\cdot) \to
      \sigma(\cdot,f(\cdot),f(\cdot-\tau))\varphi(\cdot))$
      with respect to the weak topology in $L^{2}([0,T];\R^{d})$ as $n \to \infty$, which along with Lemma \ref{lem:compactG} implies
      $$\lim_{n \to \infty}\Big|\int_{0}^{\cdot}\sigma(s,f(s),f(s-\tau))
      (\varphi_{n}(s)-\varphi(s))\diff{s}\Big|_{C([0,T];\R^{d})} = 0.$$
      It follows that $\Psi_{t}$ is continuous. Since $\{(Y^{\varepsilon_{n},u^{\varepsilon_{n}}},
      u^{\varepsilon_{n}})\}_{\varepsilon_{n} > 0}$ converges in distribution to $(\widetilde{X},\widetilde{u})$ and $\Psi_{t}$ is bounded and continuous, we have
      \begin{equation*}
            \lim_{n \to \infty}
            \E[\Psi_{t}(Y^{\varepsilon_{n},
            u^{\varepsilon_{n}}},u^{\varepsilon_{n}})]
            =
            \E_{\widetilde{\P}}
            [\Psi_{t}(\widetilde{X},\widetilde{u})].
      \end{equation*}
      By \eqref{eq:Yvaruvart}, the H\"{o}lder inequality,
      It\^{o} isometry, \eqref{eq:bsigmasuperlinear} and Lemma \ref{lem:momentbound},
      \begin{align*}
            \E[\Psi_{t}(Y^{\varepsilon_{n},
            u^{\varepsilon_{n}}},
            u^{\varepsilon_{n}})]
            =&~
            1 \wedge \E\Big[\Big|
            \int_{0}^{t}\sqrt{\varepsilon_{n}}
            \sigma(s,Y^{\varepsilon_{n},
                     u^{\varepsilon_{n}}}(s),
                     Y^{\varepsilon_{n},
                     u^{\varepsilon_{n}}}(s-\tau))
            \diff{W(s)}\Big|\Big]
      \\\leq&~
            \Big(\E\Big[\Big|
            \int_{0}^{t}\sqrt{\varepsilon_{n}}
            \sigma(s,Y^{\varepsilon_{n},
                     u^{\varepsilon_{n}}}(s),
                     Y^{\varepsilon_{n},
                     u^{\varepsilon_{n}}}(s-\tau))
            \diff{W(s)}\Big|^{2}\Big]
            \Big)^{\frac{1}{2}}
      \\=&~
            \sqrt{\varepsilon_{n}}\Big(
            \int_{0}^{t}\E\big[|
            \sigma(s,Y^{\varepsilon_{n},
                     u^{\varepsilon_{n}}}(s),
                     Y^{\varepsilon_{n},
                     u^{\varepsilon_{n}}}(s-\tau))|^{2}
            \big]\diff{s}
            \Big)^{\frac{1}{2}}
      \\\leq&~
            K_{6}\sqrt{\varepsilon_{n}}
            \Big(
            \int_{0}^{t}1
                +
                \E[|Y^{\varepsilon_{n},
                u^{\varepsilon_{n}}}(s)|^{2q}]
                +
                \E[|Y^{\varepsilon_{n},
                u^{\varepsilon_{n}}}(s-\tau))|^{2q}]
            \diff{s}
            \Big)^{\frac{1}{2}}
      \\\leq&~
            K_{6}C\sqrt{T\varepsilon_{n}},
      \end{align*}
      which results in
      \begin{equation*}
            \lim_{n \to \infty}
            \E[\Psi_{t}(Y^{\varepsilon_{n},
            u^{\varepsilon_{n}}},u^{\varepsilon_{n}})]
            =
            \E_{\widetilde{\P}}[\Psi_{t}(\widetilde{X},\widetilde{u})]
            =
            0.
      \end{equation*}
      According to the definition of $\Psi_{t}$, $\widetilde{X}$ satisfies \eqref{eq:widetildeXt} $\widetilde{\P}$-almost surely for all $t \in [0,T]$. Since $\widetilde{X}$ has continuous paths, it follows that $\widetilde{X}$ satisfies \eqref{eq:widetildeXt} for all $t \in [0,T]$ $\widetilde{\P}$-almost surely. That is to say, we have
      \begin{equation*}
            \widetilde{X}
            =
            F\Big(\int_{0}^{\cdot}
            \widetilde{u}(s) \diff{s}\Big)
            =
            z^{\widetilde{u}},
            \quad
            \widetilde{\P}\text{-a.s.}
      \end{equation*}
      Since $(Y^{\varepsilon_{n},u^{\varepsilon_{n}}},
      u^{\varepsilon_{n}}) \xrightarrow[\varepsilon_{n} \to 0]{d} (\widetilde{X},\widetilde{u})$, then $u^{\varepsilon_{n}} \xrightarrow[\varepsilon_{n} \to 0]{d} \widetilde{u}$. This together with $u^{\varepsilon} \xrightarrow[\varepsilon \to 0]{d} u$ shows that $\widetilde{u} \overset{d}{=} u$ and consequently $z^{\widetilde{u}} \overset{d}{=} z^{u}$. Therefore,
      $$
            (Y^{\varepsilon_{n},u^{\varepsilon_{n}}},
            u^{\varepsilon_{n}})
            \xrightarrow[\varepsilon_{n} \to 0]{d}
            (z^{u},u).
      $$
      Repeating the above procedure, one has that for any subsequence $\varepsilon_{m} \to 0$ (as $m \to \infty$), there exists some subsubsequence $\varepsilon_{m_{k}} \to 0$ such that
      $$
            (Y^{\varepsilon_{m_{k}},u^{\varepsilon_{m_{k}}}},
            u^{\varepsilon_{m_{k}}})
            \xrightarrow[\varepsilon_{m_{k}} \to 0]{d}
            (z^{u},u),
      $$
      which finally implies
      $(Y^{\varepsilon,u^{\varepsilon}},u^{\varepsilon})
      \xrightarrow[\varepsilon \to 0]{d} (z^{u},u)
      $.
      Thus the proof is complete.
\end{proof}

Now we are in a position to present the proof of our main result in Theorem \ref{thm:ldpresult}.

\begin{proof}[Proof of Theorem \ref{thm:ldpresult}]
       For each $\alpha \in (0,\infty)$, the compactness of the set
       $\{z^{\varphi} | \varphi \in S_{\alpha}\}$ in $C_{\phi}([-\tau,T];\R^{d})$ is due to Lemma \ref{lem:ratefunction}.
       Assume that $\{u^{\varepsilon}\}_{\varepsilon > 0} \subset \mathcal{A}_{\alpha}$ for some $\alpha \in (0,\infty)$ and $u^{\varepsilon} \xrightarrow[\varepsilon \to 0]{d} u$ as $S_{\alpha}$-valued random variables, then
       $Y^{\varepsilon,u^{\varepsilon}} \xrightarrow[\varepsilon \to 0]{d} z^{u}$ follows from Lemma \ref{lem:ldpresult}.
       By Lemma \ref{lem:sufficientlemma}, we immediately complete the proof.
\end{proof}


\end{document}